\documentclass[paper=a4, fontsize=11pt]{scrartcl}

\usepackage{amsmath,amsfonts,amsthm} % Math packages
\usepackage{bm} 
\usepackage{dsfont}
\usepackage[pdftex]{graphicx}	
\usepackage{enumitem}
\usepackage[numbers]{natbib}

\usepackage{url}
\usepackage{colortbl}
\usepackage{booktabs}
\usepackage{tabularx} 
\usepackage{float} 
\usepackage{multirow}%
\usepackage{fullpage}
\usepackage{natbib}
\usepackage{subcaption}  %create multi-part figures
\usepackage{graphicx}
\usepackage{tabularx} % for getting wide tables to fit properly
\usepackage{float} % Get access to [H] argument for figure placement
\usepackage{caption}

\usepackage{algorithm}
\usepackage{algpseudocode}
\usepackage{adjustbox}
\usepackage{xr-hyper} 

\usepackage[colorlinks=true,citecolor=blue,urlcolor=blue]{hyperref}
\usepackage{natbib}

\usepackage{graphicx}
\usepackage[english]{babel}
\usepackage{enumitem}
\usepackage{dsfont}
\usepackage{amsmath,amssymb}

\newcommand{\diam}{\operatorname{diam}}

\newcommand{\pr}{\mathbb{P}}

\newcommand{\argmin}{\operatornamewithlimits{\arg\min}}
\newcommand{\ind}{\mathds{1}}

\newcommand{\PP}{\mathbb{P}}
\newcommand{\RR}{\mathbb{R}}

\newcommand{\EE}{\mathbb{E}}
\newtheorem{theorem}{Theorem}
\newtheorem{definition}[theorem]{Definition}

\newtheorem{corollary}[theorem]{Corollary}
\newtheorem{proposition}[theorem]{Proposition}
\newtheorem{lemma}[theorem]{Lemma}
\newtheorem{example}{Example}

\usepackage[T1]{fontenc}

\usepackage[english]{babel}															% English language/hyphenation
\usepackage{amsmath,amsfonts,amsthm} % Math packages
\usepackage{bm} 
\usepackage{dsfont}
\usepackage[pdftex]{graphicx}	
\usepackage{enumitem}

\usepackage{url}
\usepackage{colortbl}
\usepackage{booktabs}
\usepackage{tabularx} 
\usepackage{float} 
\usepackage{multirow}%
\usepackage{fullpage}
\usepackage{natbib}
\usepackage{graphicx}
\usepackage{tabularx} % for getting wide tables to fit properly
\usepackage{float} % Get access to [H] argument for figure placement
\usepackage{caption}

\usepackage{xr-hyper} 

\usepackage[colorlinks=true,citecolor=blue,urlcolor=blue]{hyperref}
\usepackage{natbib}

\usepackage{chngcntr}
\counterwithout{table}{section}
\counterwithout{figure}{section}
\newcommand{\horrule}[1]{\rule{\linewidth}{#1}} 	% Horizontal rule

\newcommand{\rev}[1]{\textcolor{black}{#1}}
\title{
		\vspace{-1in} 	
		\usefont{OT1}{bch}{b}{n}
		\horrule{0.5pt} \\[0.4cm]
		\Large 
        Revisiting local regression: shape regularity, uniform rates, and the limits of random splits
		\\
		\horrule{2pt} \\[0.1cm]
}

\date{\large \today}

\author{ \textbf{Jérémy Bettinger, François Portier, Adrien Saumard} \\
\large \href{mailto:jeremy.bettinger@ensai.fr}{jeremy.bettinger@ensai.fr} ; \href{mailto:francois.portier@ensai.fr}{francois.portier@ensai.fr} ; \href{mailto:adrien.saumard@ensai.fr}{adrien.saumard@ensai.fr} \\
\large Department of Statistics, \\
\large University of Rennes, ENSAI, CNRS, CREST-UMR 9194, F-35000 Rennes, France }

\begin{document}

\maketitle



\begin{abstract}

% --- ORIGINAL ABSTRACT (kept for comparison) ---------------------------------
% Considering pointwise and sup-norm estimation, we analyze the non-asymptotic behavior of various non-parametric local averaging estimators for Lipschitz regression functions. We first establish a general deviation bound for estimators based on a VC family of localizing sets in the covariate space. We then introduce the concept of shape-regular local maps, in which local averaging is performed over sets with an ``almost isotropic'' geometry. We show that shape regularity is a fundamental cornerstone of statistical efficiency: it is both necessary and sufficient to achieve optimal rates, up to logarithmic factors. Turning to specific examples, we establish that (i) the simple $k$-nearest neighbor method achieves optimal error rate, in contrast to (ii) trees following random split construction. To bridge this gap, we propose an alternative tree construction that explicitly incorporates shape regularity constraints, thereby ensuring optimal estimation of Lipschitz functions.
% --- REWRITTEN ABSTRACT (framing pass) ---------------------------------------
\rev{Considering pointwise and sup-norm estimation, we analyze the non-asymptotic behavior of local averaging estimators for Lipschitz regression functions. Building on a general deviation bound for estimators based on a VC family of localizing sets, we introduce the notion of \emph{shape-regular} local maps, where averaging is performed over sets with an almost isotropic geometry. Our main message is a characterization: shape regularity is both \emph{necessary and sufficient} to attain optimal rates, up to logarithmic factors. Necessity is established non-asymptotically through an explicit anisotropic example, sharpening a phenomenon previously understood only heuristically in asymptotic theory. We then draw two consequences. First, the simple $k$-nearest neighbor rule is shape-regular by construction and attains the optimal rate, even on unbounded supports. Second, and perhaps surprisingly, the popular random-split condition for trees -- known to ensure consistency and vanishing cell diameters -- does \emph{not} guarantee optimal rates: for blind tree constructions, the cell aspect ratio diverges exponentially with depth, so that shape regularity fails with positive probability. This identifies the absence of a geometric correction mechanism, rather than a slowly shrinking diameter, as the obstruction to optimality. Motivated by this gap, we propose a tree construction that enforces shape regularity through a simple constraint on admissible splits, and prove a uniform deviation inequality showing that it restores the optimal rate for Lipschitz functions.}

\end{abstract}

\section{Introduction}\label{s1}

Consider the standard regression problem where the goal is to estimate the regression function of a random variable $Y\in \mathbb R$ given the covariates vector $X\in \mathbb R^d$, defined as $ g(x) := \mathbb E [ Y| X= x]$, $x\in \mathbb R ^d$. One leading approach, called \textit{local regression} or \textit{local averaging}, consists in averaging the observed response variables, restricted to covariates that lie in a small region of the domain $\mathbb R ^d$. Local regression methods include kernel smoothing regression \cite{nadaraya1964estimating}, nearest neighbors algorithm \cite{fix1989discriminatory,cover1968estimation} and regression trees or, more generally, partitioning regression estimators \cite{breiman1984classification,nobel}. We refer to the books
\cite{devroye96probabilistic,gyorfi2006distribution} for an overview of local regression methods and to \cite{biau2015lectures} for a precise theoretical account on the nearest neighbors algorithm.

Concerning the estimation problem, when the error is measured in terms of the mean squared error ($L_2$-error), the optimal convergence rates are known \cite{stone1982optimal} and depend on the smoothness of the regression function $g$. Whether or not these convergence rates are achieved often serves as a theoretical baseline to evaluate the accuracy of local regression methods. For example, a Lipschitz function $g$ can only be approximated at the rate $ n ^{- 1/ (d+2)  }$ in general, when $n $ independent observations are given. Many of the above estimators are known to achieve optimal convergence rates. The nearest neighbors, the Nadaraya-Watson and the fixed partitioning (histogram) regression estimators are all optimal for Lipschitz functions (as well as for twice differentiable functions for the first two listed methods), as explained in \cite{biau2015lectures},  \cite{tsy_08} and Chapter 4 of \cite{gyorfi2006distribution}, respectively. Furthermore, the Nadaraya-Watson \cite{einmahl2000,gine+g:02} and the nearest neighbors \cite{kpotufe2011k,chaudhuri2014rates,jiang2019non,portier2021nearest} estimators are both known to achieve a rate of sup-norm convergence that is of the same order as the $L_2$-rate, up to a logarithmic term.

Regression trees \cite{breiman1984classification} occupy a central place among local methods as they provide a data-driven recursive partition of the feature space, forming the basic building blocks of modern ensemble methods such as random forests \cite{breiman2001random,biau2016random}. A particularly tractable class of tree constructions are those where the split mechanism is chosen independently of the observed responses. Although less common in practice than impurity-based CART splits \cite{breiman1984classification}, such data-independent splitting schemes \cite{biau2012analysis,biau2016random} provide a clear mathematical framework for isolating the effect of tree depth and partition geometry on statistical performance. Under the \textit{random split condition}, the splitting direction at each step is selected randomly and independently of the sample, with every direction being chosen with strictly positive probability. This condition is employed, for instance, in \cite{meinshausen2006quantile} for quantile estimation, in \cite{wager2014asymptotic,wager2018estimation} for heterogeneous treatment effect estimation with inferential guarantees, and in \cite{biau2012analysis,duroux2018impact} for regression estimation. Under this condition, the diameter of each cell can be shown to shrink to zero \cite{meinshausen2006quantile,biau2012analysis} at a certain rate \cite{wager2018estimation,duroux2018impact}, which is key to establishing the consistency \cite{meinshausen2006quantile,biau2012analysis} as well as error bounds \cite{wager2018estimation,duroux2018impact} for the resulting estimates. Two notable constructions within this framework are the median forest and the centered forest. The median forest \cite{duroux2018impact} selects the split coordinate uniformly at random and cuts at the median of the chosen side, while the centered forest \cite{biau2012analysis} cuts at the midpoint. The former achieves a better convergence rate than the latter. However, neither attains the optimal rate for Lipschitz functions.
%In \cite{zhang2024adaptive}, a different side selection mechanism is proposed: the split is always assigned to whichever side has been chosen least often. The resulting rate is optimal for Hölder functions, suggesting that the randomness inherent in the random-split condition may be unnecessary, or at least suboptimal. All the results mentioned above are upper bounds; no matching lower bound is currently known.

%(known results on CART-like regression trees.)
%
Despite the many existing results available for the Nadaraya-Watson and nearest neighbors regression estimators, and also fixed or purely random partitioning regression rules, \rev{little is known} about local regression based on data-dependent partitions, such as the well-known CART regression tree \cite{breiman1984classification}. Such an algorithm is indeed much harder to analyze mathematically.
First results on data dependent partitions can be found in \cite{stone1977consistent}, but they are restricted to cases where the partition depends only on the covariates, as in nearest neighbors regression or for statistically equivalent blocks \cite{anderson}. 
%  In such a case, one might work conditionally on the covariates and rely on standard tools for  sums of independent random variables.  
More advanced results, that are valid for general data dependent partitioning estimators, are obtained in \cite{gordon1980consistent,breiman1984classification,nobel}, where conditions are given to ensure almost sure $L_2$-consistency. The typical assumptions that are required in the previous works include (i) large enough points in each partition element and (ii) small diameter, while having (iii) a reduced complexity on the partition elements. Note also that Theorem 1 in \cite{scornet2015consistency} can be applied to CART regression algorithm and gives sufficient conditions for the $L_2$-consistency.

Beyond consistency, \rev{little is known} about the convergence rates of data-dependent, CART-like regression tree estimators. Recent studies  \cite{chi2022asymptotic, mazumder2024convergence} have obtained convergence rates for the $L_2$-error under the so-called \textit{sufficient impurity decrease} (SID) condition,  a restrictive assumption on the splitting rule that is not always satisfied in practice. The rate of convergence depends on a parameter -- denoted $\lambda$ in \cite{ mazumder2024convergence} -- quantifying the strength of the SID condition, and it is not \textit{a priori} easy to discuss the rate optimality. %Nonetheless, it is shown in  \cite{ mazumder2024convergence} that for a univariate linear regression function, the rate obtained through the SID condition is actually optimal. A specific class of additive regression functions achieving a particular smoothness assumption called the ``locally reverse Poincaré inequality'' is provided in  \cite{ mazumder2024convergence}, satisfying the SID condition. 
In another direction, the recent negative results in \cite{cattaneo2022pointwise} show that CART regression can be sub-optimal, and even inconsistent, for the pointwise -- and also uniform -- estimation error. Such phenomenon does not occur when focusing on the $L_2$-error, but as highlighted in \cite{cattaneo2022pointwise}, pointwise convergence of decision trees is also essential for reliability of the methodologies developed in some causal inference and multi-step semi-parametric settings for instance.

%
%(our results)
%
\rev{Despite this rich literature, a basic question remains without a non-asymptotic answer: \emph{what geometric property of the localizing sets is responsible for optimal pointwise and uniform rates?} Asymptotic theory has long suggested that the cells should be ``well-shaped'' \cite{gyorfi2006distribution}, but to our knowledge no result establishes such a property as both necessary and sufficient with explicit, finite-sample rates. This question is not merely theoretical: the recent negative results in \cite{cattaneo2022pointwise} show that widely used recursive partitioning schemes can be sub-optimal, or even inconsistent, precisely for the pointwise and uniform errors that matter in causal and semi-parametric applications. Our aim is to isolate the geometric condition that separates optimal from sub-optimal local averaging, and to show that it has concrete algorithmic consequences.}

In this work, we develop a theory for obtaining pointwise and uniform rates of convergence for a large class of local regression estimators, that includes previously mentioned partitioning estimators. More precisely, in a random design regression with heteroscedastic sub-Gaussian noise framework, the theory allows the localization method to be general, in the sense that it may depend on a different source of randomness or on the covariates sample (as for nearest neighbors) and even on the full regression sample (as in CART). 

We first obtain a general probability upper bound (Theorem \ref{th:general}) for the pointwise estimation error of any estimator that is based on a VC class of localizing sets. In contrast to the $L_2$-error bound \cite{lugosinobel}, where the combinatorial size of the class of all partitions must be controlled, focusing on the pointwise error allows to invoke the Vapnik dimension of the \textit{elements} of the partition.
%We indeed point out that the major advantage of focusing on the pointwise error, compared to the $L_2$-error, is that it allows the use of the Vapnik dimension of a class containing the \textit{elements} of the random partition, instead of having to control the combinatorial size of the class of the \textit{entire} partitions themselves, as in \cite{lugosinobel}. 
Our bound reveals a trade-off between the cell diameter and its empirical measure, recovering, in a non-asymptotic framework, the essence of classical conclusions from asymptotic theory \cite{gyorfi2006distribution}. This trade-off is further analyzed through the property of \textit{shape regularity}, that requires the localizing sets to exhibit isotropic geometry, which is shown, with the help of an example, to be necessary for reaching the optimal error bound. We also show, for general local maps estimator that shape regularity is actually sufficient to obtain optimal error bound (Theorem \ref{th2:general}).

%Furthermore, our result reveals that achieving optimal rates is closely tied to the geometric control of the partition (Section \ref{sectionSR}). 
%The first result that is established is an upper bound - valid with high probability - on the uniform error of local regression estimators. The bound involves two terms: 
 %In that, we extend the asymptotic results obtained in  \cite{gordon1980consistent,breiman1984classification,nobel,scornet2015consistency} where the $L_2$-consistency is obtained (without rates of convergence). 

%A central result, proven in Section \ref{s30} and derived from Vapnik-Chervonenkis inequalities and sub-Gaussianity, Theorem \ref{th:general}, provides an upper bound on the estimation error by . This fundamental trade-off reveals that achieving optimal rates is closely tied to the geometric control of the partition (Section \ref{sectionSR}). 

%In Section \ref{s_appli}, 
We then examine several applications of our theory, focusing on the shape regularity property to characterize the error bound attained in each case.

\begin{itemize}
\item[(i)] 
%how minimal mass assumptions and shape-regular local maps allow estimators to reach minimax optimality, 
We revisit the classic $k$-nearest neighbors algorithm establishing optimal error bound that extend some recent results \cite{jiang2019non,portier2021nearest} to unbounded covariate support using the so called \textit{strong minimal mass assumption}  \cite{gadat2016classification}.  
    \item[(ii)]  As detailed before, many theoretical results from regression tree and random forest literature rely on  the random split condition \cite{biau2012analysis,duroux2018impact,wager2018estimation}, i.e., any direction can be split with positive probability.
%While we prove the consistency of a subclass of these Wager-type trees, we highlight the insufficiency of standard Wager conditions to guarantee shape regularity. %The geometric integrity of these trees rests on two key hypotheses: the random-split property and $\alpha$-regularity. The former ensures that every feature has a non-vanishing probability of being selected, providing comprehensive coverage of the feature space. While this ensures the consistency of certain Wager-type trees, it often leads to non-adaptive splits along irrelevant dimensions, "wasting" model complexity compared to more greedy approaches like CART. In parallel, $\alpha$-regularity prevents the formation of degenerate leaves by maintaining a minimum fraction of data in each node. Regarding Wager-type trees, 
After establishing a rate of convergence for such random split tree,  %we also prove that there exist, with positive probability, Wager-type trees satisfying both the random-split and $\alpha$-regularity hypotheses which fail to be shape regular. 
we show that for ``blind'' tree constructions -- characterized by the independence between the split directions and the split positions -- the aspect ratio between the largest and smallest sides of a cell diverges as the depth increases, leading to sub-optimal convergence rates.
% --- ORIGINAL (kept for comparison) ---
% This provides a fundamental insight: %these are fundamentally distinct concepts, as 
% the popular random split condition is insufficient to prevent the geometric degradation of the cells and thus fail to guarantee optimal convergence rates.  %This provides a fundamental insight: to achieve optimal rates, the underlying tree construction must be adaptive in some manner.
% --- REWRITTEN (framing pass) ---
\rev{This is, perhaps, counter-intuitive: the random-split condition is strong enough to force the cell diameters to zero -- hence consistency -- yet too weak to control their \emph{shape}, so that optimality fails. The obstruction is not an insufficiently small diameter, as one might expect, but the unchecked elongation of the cells.}

%The previous two points support the claim that shape regularity is necessary %(Proposition \ref{contre_ex}) 
%and sufficient 
%(Theorem \ref{main_result_localizing_map}) 
%to obtain optimal rates -- up to logarithmic factors -- for the pointwise and uniform estimation errors. 
\newpage

\item[(iii)] The above two points reveal a tension: a simple regression rule such as $k$-NN achieves optimal rates, whereas a tree with a more complex rule fails to do so. This motivates exploring splitting rules beyond that of the random split. This study culminates in the proposal of a new tree construction that explicitly incorporates shape regularity constraints, ensuring geometric stability while preserving data-driven adaptivity. We derive a deviation inequality for the uniform estimation error of those shape regular trees, grown by enforcing a minimum number of points per leaf and a simple rule maintaining the shape regularity of the localizing sets. %In the case of partitions made of hyper-rectangles, such as for CART-like algorithms, the shape-regularity condition reduces to a control of the largest side length of the localizing set by its smallest side length. Recent results obtained in \cite{cattaneo2022pointwise} indeed tend to indicate that such rules additions are likely to be unavoidable to ensure good pointwise convergence rates of CART-like regression trees. 
%\color{red}Optimal rates have been obtained recently in \cite{zhang2024adaptive} using a different approach. Instead of the usual random split mechanism, the authors propose to split the less split coordinate. This is different from our approach that allows any sample-based split as long as this one does not alter shape regularity.
\color{black}
\end{itemize}

% It is worth noting that our approach substantially differs from the use of the SID condition \cite{ mazumder2024convergence} described earlier. The latter indeed ensures convergence rates for the $L_2$-error and is highly linked to the precise cost in the splitting rule of CART, defined through the so-called impurity gain. Moreover, the SID condition is expressed through the behavior of the unknown regression function and covariates distribution, and cannot hold for any regression function. In contrast, our shape-regularity condition does not depend on the regression function $g$, neither on the covariates distribution, and only imposes a restriction that may be effective with any cost function involved in the splitting rule. This makes our shape regularity condition easy to guarantee in practice as illustrated in Algorithm \ref{alg:cart-like} (see Section \ref{s53}), where a general cost function is used to build the tree.

%Through this progression, shape regularity is revealed as the pivotal link between the geometric structure of a partition and its statistical efficiency.

The outline is as follows. We state in Section \ref{s2} some necessary background and formulate the setting of local regression map estimators. 
 Section \ref{s30} then gives a first deviation inequality for local regression map estimators. Section \ref{sectionSR} introduces the shape regularity property and reveals its importance to obtain optimal error bound. Section \ref{s_appli} covers the three applications described in (i), (ii), and (iii) above, treated respectively in Sections  \ref{sec_k_nn}, \ref{sec_wag}, and \ref{s53}.
%beginning in Section \ref{sec_wag} with Wager-type trees, %where we compare Wager's inherent regularity ($\alpha$-regularity and directional constraints) with our notion of shape regularity. Furthermore, we emphasize the critical importance of algorithmic adaptivity, demonstrating how underlying dependencies in the splitting process are essential for maintaining geometric balance and achieving optimal performance. Next, 
%following  by Section \ref{sec_SR_opt} is dedicated to pointwise and uniform convergence bounds for data-dependent regression maps under shape regularity constraints, including an application to $k$-nearest neighbors. We conclude with a final application to CART in Section \ref{s53}. 
Section \ref{sec:perspective} offers some perspectives for further research to overcome the dimensionality curse of regression trees. All the mathematical proofs are given in the Appendix.

\section{Mathematical background}\label{s2}

\subsection{Regression set-up}\label{s21}

Let $(X,Y)$ be a random vector with probability distribution $ P$ on $\mathbb R^d \times \mathbb R$, where $d\geq 1$ is the dimension of covariates vector $X\in S_X \subset\mathbb R^d$ and $Y \in \mathbb R$ is the output variable.  The goal is to estimate the conditional expectation $x\mapsto g(x) = \mathbb E [ Y|X= x]$, $x\in S_X$. The quality of the estimation of the function $g$ by an estimator $\hat g$ will be assessed with the help of the uniform norm defined as $\sup_{x\in S_X} | \hat g(x) - g(x) | $. %We also further assume that $S_X$ is open and convex to allow to differentiate $g$ in the classical sense. 
For a fixed $x\in S_X$, we also address the estimation error of the value $g(x)$ through the analysis of the deviations of the quantity $| \hat g(x) - g(x) | $. 

The following assumption on $ P$ will be key in this work and, roughly speaking, amounts to assume that the noise $\varepsilon = Y  - g(X)$ in the regression model is lightly tailed.

\begin{enumerate}[label=(E), wide=0.5em,  leftmargin=*]
\item \label{cond:epsilon} 
The random variable $\varepsilon$ is sub-Gaussian conditionally on $X$ with parameter $\sigma^2$. That is, $\mathbb E [\varepsilon |X ]= 0$ and for all $\lambda \in \mathbb R$, $$ \mathbb E [ \exp( \lambda\varepsilon  ) |X ] \leq \exp\left( \frac{ \lambda ^2 \sigma^2}{ 2 }\right).$$
\end{enumerate}
Note that under assumption \ref{cond:epsilon}, the noise term $\varepsilon$ is squared integrable and it is allowed to depend on the covariates $X$.  In particular, the noise is \textit{heteroscedastic}, with a uniform upper bound on its conditional variance: almost surely, we have $\mathbb{E}[\varepsilon^2\vert X] \leq \sigma^2$. A more restrictive assumption is when $\varepsilon$ is independent of $X $ and sub-Gaussian with parameter $\sigma^2$.

%Associated to the regression model, 

A real function $h$ on $S_X$ is called $L$-Lipschitz as soon as $ | h(x) - h(y) |\leq L \|x-y\|_2$ for all $(x,y)\in S_X^2$.
  In what follows, we will consider regression functions that are Lipschitz over the domain $S_X$:
\begin{enumerate}[label=(L), wide=0.5em,  leftmargin=*]
  \item \label{cond:reg4} The function \( g: x \mapsto \mathbb E [ Y|X= x]\) is $L$-Lipschitz on \( S_X \). 
\end{enumerate}
Define also the local Lipschitz constant $ L(V)$ of $h$ over $V\subset S_X$ as the smallest constant $L>0$ such that, for all $(x ,y)$ in $V^2$,
$$     |h(x) - h(y) | \leq L  \|  x-y\|_2.$$ For a $L$-Lipschitz function, it holds $L(V)\leq L$ for any set $V\subset S_X$.

In this work, all the estimators will be based on the sample $\mathcal{D}_n=\left\{(X_i,Y_i) \,  : \, i=1,\ldots,n\right\}$ which satisfies the following assumption:

\begin{enumerate}[label=(D), wide=0.5em,  leftmargin=*]
\item \label{cond:D} 
The random variables $\{(X,Y), (X_i,Y_i)_ {i=1,\ldots,n} \}$ are independent and identically distributed with common distribution $  P$. 
\end{enumerate}
Let us introduce the notation $P^X$ as the marginal distribution of $X$. Set also $\varepsilon_ i := Y_i - g(X_i) $ for each $i=1,\ldots, n$. 
%Let us further define some quantities that will be instrumental in our analysis. 
In the following, $\lambda(V)$ denotes the Lebesgue measure of any set $V$, which we simply refer to as its volume. Moreover, for any set $V$, its diameter is given by the formula
$$ \diam (V)  = \sup_{(x,y)\in V\times V} \|x-y\|_2,$$
where $\|x\|_2^2 = \sum_{k=1} ^d x_k^2$. %Note that since all norms are equivalent in finite dimension, any diameter may be chosen, up to a constant factor depending on a power of $d$. 
The closed ball with center $x\in \mathbb R^d$ and radius $r>0$ is denoted by $B(x,r)$.

\subsection{Local regression maps}\label{s22}

We consider general local regression estimators using the concept of local maps so as to include regression trees and partitioning estimators but also the nearest neighbors regression rule. Let $\mathcal B(S_X)$ denote the Borel $\sigma$-algebra on $S_X$.

\begin{definition}
    A local map for a variable $X$ is a mapping $ \mathcal V : S_X \to \mathcal B(S_X) $ such that for all $x\in S_X$, $x\in \mathcal V (x)$. 
\end{definition}

For any local map $\mathcal V$, the associated regression estimator is given by
\begin{equation*}
    \forall x \in S_X, \quad \hat g_{\mathcal V}(x) = \frac{\sum_{i=1} ^ n Y_i\mathds 1 _{ \mathcal V(x) }(X_i) }{\sum_{i=1} ^ n \mathds 1 _{ \mathcal V(x) }(X_i)  }\; , 
\end{equation*}
with the convention that $0/0 = 0$, which is in force in the subsequent work. %This work focuses on continuous covariates, therefore all local maps will have their images to sets with positive Lebesgue measure. 
Local maps  $\mathcal V $ depending on the sample $(X_1,Y_1),\ldots , (X_n, Y_n)$ are of particular interest. This is indeed the case for some adaptive tree constructions, as well as for the nearest neighbors algorithm. Let us also stress out that similar maps were introduced in \cite{nobel}, where they are however restricted to partition based estimator.  %We list now some examples that are included in the previous framework. Each is associated to a different map $\mathcal V$ and more particularly to a different dependence structure with respect to the sample. %These contrasting properties between the estimators are also described in \cite{JMLR:v9:biau08a,biau2012analysis}.

The local regression map framework is particularly interesting because it includes a variety of different methods, e.g., fixed partitioning,  purely random trees, nearest neighbors, and CART-like constructions, and each method induces a particular dependence structure when creating the partition.

\begin{example}[fixed hyper-rectangles partition]\label{ex1}
    The most simple case for the dependence structure of the local map is when the partition is fixed, not random. Suppose $S_X= (0,1]^d$. For each coordinate $k =1,\ldots, d$, consider the collection $0 = u_0^{(k)} < u_1^{(k)} < \ldots < u_{N_k}^{(k)} = 1$. This allows to introduce a partition of $S_X$ made of $ \prod_{k=1} ^d N_k $ elements defined as $ V_{i_1,\ldots, i_d} =  \prod_{ k = 1 }^d (u_{i_k}^{(k)}, u_{i_k + 1 }^{(k)} ]$ for each d-uplet $(i_1,\ldots, i_d)$ satisfying $i_\ell\in \left\{0,\ldots,N_\ell-1 \right\}$ for $\ell\in \left\{1,\ldots,d \right\}$. Note that each $V_{i_1,\ldots, i_d}$ has a positive Lebesgue measure $ \prod_{ k = 1 }^d (u_{i_k + 1 }^{(k)} - u_{i_k}^{(k)} )$.
\end{example} 

\begin{example}[purely random trees]
In contrast to Example \ref{ex1}, a \textit{purely random tree} construction, as described in \cite{arlot2014analysis} and initially introduced in \cite{breiman2000some},  consists in using some randomness that is independent of the observed sample. It includes centered (resp. uniform) trees, for which the split direction is uniformly distributed along the space coordinates and the split location of the selected side is at the center (resp. uniformly distributed). It also includes Mondrian trees \citep{lakshminarayanan2014mondrian}, where the split direction is selected at random depending on the shape - i.e. side lengths - of the leaf. This will be explored in a forthcoming article. % (with the difference that also the leaf to be split is selected at random). %Random forest estimators resulting from such trees are studied in \cite{JMLR:v9:biau08a}. 
\end{example}

\begin{example}[nearest neighbors regression]\label{ex3}
    Nearest neighbors algorithm induces a Voronoi-like partition, which dependence structure is different from the one of purely random trees, since the nearest neighbors partition depends on the data through the location of the covariates in the space. The $k$-nearest neighbors ($k$-NN) estimator (see \cite{biau2015lectures} for a recent textbook) is defined, for each $x\in S_X$, as the average responses among the $k$-nearest neighbors to point $x$. As such, we have
$$ \hat g_{NN}(x) =  \frac{1}{k }  \sum_{i=1}^n  Y_i\ind_{B(x, \hat \tau_k(x))} (X_i ) ,$$ where $\hat \tau_k(x)$ is the so-called $k$-NN radius defined as the smallest radius $\tau>0 $ such that $k \leq\sum_{i=1}^n  \ind_{B(x,  \tau ) } (X_i )$. Note that here the local map is $\mathcal V (x)  = B(x, \hat \tau_k(x))$ and therefore  depends on $X_1,\ldots, X_n$.
\end{example}

\begin{example}[{CART-like trees}]
Regression trees are a class of partition based estimators where the partition is recursively built, and made of hyper-rectangles. Therefore, they are part of the local map framework, just as examples 1 and 2 above. Usual regression trees are grown sequentially by splitting stage-wise each (adult) leaf into two (children) leafs. In most cases, as in CART regression \citep{breiman1984classification}, each cell division results from splitting along one single variable according to a data-based criterion. This precise step is crucial as it allows to adapt the partition to the prediction problem. For instance, if one variable is not significant then it must be better not to split with respect to it. This enables to obtain a flexible regression estimator, which behaves well in many problems even when the dimension $d$ is rather large. The fact that the resulting partition depends on the full data (including the response) is however problematic for the theory since in this case, the local averaging estimator is not a sum over independent random variables, thus prohibiting a direct application of concentration inequalities for sums of independent observations. Finally, it is worth mentioning that CART regression trees are the ones that are usually combined in the standard Random Forest regression algorithm, as introduced in \cite{breiman2001random}.
 \end{example}

\section{A deviation bound for local map estimators}\label{s30}
Considering the local map estimator definition given in Section \ref{s22}, the first step in analyzing its pointwise error is standard, and consists in considering the following bias-variance decomposition,
\begin{equation*}
    \hat g_{\mathcal V}(x) - g(x) = \underbrace{\frac{\sum_{i=1}^n \varepsilon_i \mathds 1 _{ \mathcal V(x) }(X_i)}{\sum_{j=1}^n \mathds 1 _{ \mathcal V(x) }(X_j)}}_{\text{variance term}}+\underbrace{\frac{\sum_{i=1}^n \left(g(X_i) - g(x)\right) \mathds 1 _{ \mathcal V(x) }(X_i)}{\sum_{j=1}^n \mathds 1 _{ \mathcal V(x) }(X_j)}}_{\text{bias term}}.
\end{equation*}

In this section, we shall first provide a preliminary concentration bound for the variance term, which is free from any restriction on the covariate distribution. Subsequently, we leverage this result in the regression framework to obtain a concentration bound for the estimation error.

%\subsection{A deviation bound for the variance term}\label{3.1}

The \textit{shattering coefficient}, as introduced in Vapnik's seminal work \cite{vapnik2015uniform} and  detailed for example in \cite{wellner1996,devroye96probabilistic},
is key to obtain upper bounds on certain empirical sums indexed by sets or functions. Let $\mathcal A$ be a collection of subsets of a set $ S$. Given an arbitrary collection $z =( z_1,\ldots, z_n)$ of distinct points in $  S$, consider the collection of $\mathbb R^n$-points
$ \ind _ {\mathcal A} (z)$ defined as $  \{ (\ind _ A (z_1) \ldots, \ind_A (z_n )) : A\in \mathcal A \}\subset \{0,1\}^n $. 
We have that $|\ind _ {\mathcal A} (z)  | \leq 2^n$ and when $   |\ind _ {\mathcal A} (z) |  = 2^n$ we say that $z$ is shattered by $\mathcal A$. An important quantity is then
 $$\mathbb S_\mathcal A(n) : =  \sup_{z\in \mathbb R^n} | \ind _ {\mathcal A} (z) | $$
which is called the shattering coefficient. %It represents loosely speaking the number of possible $\ind _ {\mathcal A}$.

We now provide a VC-type inequality tailored to the analysis of the variance term for local regression estimators. Recall that, by convention, $0/0=0$.

\begin{theorem}\label{1}
Let $n\geq 1$ and $\delta \in (0,1) $. Suppose that \ref{cond:epsilon} and \ref{cond:D} are fulfilled and that $\{\mathcal V (x) \, :\, x\in \mathbb R^d\} \subset \mathcal A$, a deterministic collection of sets in $\mathbb R^d$. The following inequality holds with probability at least $1 - \delta$,
$$\sup_{x\in \mathbb R^d} \dfrac{\sum_{i=1}^n \varepsilon_i \mathds 1 _{ \mathcal V(x) }(X_i)}{\sqrt{\sum_{j=1}^n \mathds 1 _{ \mathcal V(x) }(X_j)}} \leq \sqrt{2 \sigma^2 \log\left( \frac{\mathbb S_ {  \mathcal A   } (n)}{\delta} \right)}.$$
%$$\sup_{\mathcal{V} \in \mathcal{V}} \left|\dfrac{\sum_{i=1}^n \varepsilon_i \mathds 1 _{ \mathcal V(x) }(X_i)}{\sqrt{\sum_{j=1}^n \mathds 1 _{ \mathcal V(x) }(X_j)}} \right| \leq 2 \sqrt{2 \sigma^2 \log\left( \frac{\mathbb S_\mathcal V(n)}{\delta} \right)}.$$
\end{theorem}
Note that in Theorem \ref{1} above, only an upper bound is given but a lower bound is also valid, since the same holds true when each $\varepsilon _ i $ are replaced by $ -\varepsilon_i$. Moreover, combining such inequalities through a union bound gives a result for the supremum of the absolute value.

We now state a general deviation bound on the uniform error of local regression map estimators with finite Vapnik-Chervonenkis (VC) dimension. The VC dimension is defined as 
\begin{align*}
  vc(\mathcal A)  &= \max \{ n\geq 1 \, :\, \mathbb S_\mathcal A(n) 
 = 2^n  \}. 
\end{align*}
As a consequence, the fact that all given $z_1,\cdots , z_{v+1} $ points cannot be shattered is equivalent to the fact that the VC dimension is smaller than $v$. The reason why the VC dimension is appropriate for controlling the complexity of classes of sets is perhaps explained by the Sauer's lemma (see \cite{lugosi2002pattern} for a proof) which states that
$\mathbb{S}_{\mathcal{A}}(n)  \leq  \sum_{i=0} ^{vc(\mathcal{A})} \binom{n}{i}.$ %  (en/vc(\mathcal A)) ^{vc(\mathcal A)}
  %Note that when $n\leq vc(\mathcal A)$ the bound $2^n \leq 2 ^{vc(\mathcal A)}$ is better than the above. 
  A consequence of Sauer's lemma is that 
  $\mathbb S_\mathcal A(n) \leq (n+1)^{vc(\mathcal A) }.$
%Hence the above bound for $n $ large enough improves upon the direct $2^n $ bound.
%Consequently, the number of functions within $ \ind _ {\mathcal A} (z) $ is of order $n^{vc(\mathcal A)}$. 

As established in \cite{wenocur1981some}, previous examples include the class of cells $(-\infty , t]\subset \mathbb R^d$, having VC dimension equal to $d$, or the class $( s , t]$, $s,t\in \RR^d$, of VC dimension equal to $2d$. In addition, the class of balls in $\mathbb R^d$ has dimension equal to $d+1$. 
\begin{definition}
 A local map $\mathcal V$ is said to be VC when there exists $\mathcal A$, a fixed VC collection of sets in $\mathbb R^d$, such that $\{\mathcal V (x) \, :\, x\in S_X\} \subset \mathcal A$.
\end{definition}

The next probability error bound is valid for local map estimators, with a general VC local map, that may for instance depend on the sample.

\begin{theorem}\label{th:general}
Let $n\geq 1$ and $\delta \in (0,1/2) $. Under \ref{cond:epsilon}, \ref{cond:D} and \ref{cond:reg4}, suppose that the local map is VC with dimension $v$. We have, with probability at least $1 - 2 \delta $, for all $x\in S_X$,
   % \begin{align*}
   %    &|  \hat g_{\mathcal V}(x) - g(x)| \\
   %    &\leq \sqrt{\frac{ 2 \sigma^2 \log\left( \frac{ (n+1) ^v }%{\delta} \right)}{ n \PP_n(\mathcal V (x))  } } + \sup_{y \in %\mathcal V (x)} |g(y) - g(x)|
    %\end{align*}
        \begin{align*}
       |  \hat g_{\mathcal V}(x) - g(x)| \leq \sqrt{\frac{ 2 \sigma^2} { n P_n^X(\mathcal V (x))  } \log\left( \frac{ (n+1) ^v }{\delta} \right) } + L(\mathcal V(x)  ) \diam  (\mathcal V(x) )  .
    \end{align*}
    where for any $A\in \mathcal B(S_X)$, $n P_n^X(A)  = \sum_{i=1} ^n \ind _A(X_i) $.
\end{theorem}

%Before pointing out to some related results, let us stress that convexity of $\mathcal V$ is often met in pratcice (e.g., trees and nearest neighbor) but can still be alleviated at the price of changing the $L_p(\mathcal V(x)  )$ constant into a slightly larger constant. Details are given in the Supplementary after the proof of Theorem \ref{th:general}. 
An alternative approach proposed in \cite{lugosinobel,nobel} as well as in \cite{devroye96probabilistic}, see Theorem 21.2 therein, follows from a uniform control over all resulting partitions, implying consistency results for sums over all partition elements. In Theorem \ref{th:general}, our approach is substantially different, since by considering the pointwise or sup-norm error,  the complexity term comes from the elements of the partition only. In addition, Theorem \ref{th:general} above might be compared with Theorem 6.1 in \cite{devroye96probabilistic}, which is suitable to either non-random or purely random (i.e., independent of the sample) data partitioning \citep{JMLR:v9:biau08a}. While Theorem \ref{th:general} is valid for data dependent partitions, we recover almost sure consistency by imposing two conditions that are similar to those required in Theorem 6.1 of \cite{devroye96probabilistic}, namely $\diam  (\mathcal V(x) )  \to 0 $ and $ nP_n^X( \mathcal V(x) ) / \log(n )  \to 0$\rev{, the former driving the bias to zero and the latter ensuring that enough points fall in the cell to control the variance}. Depending on whether the previous conditions hold uniformly in $x$ or for a given $x$, the consistency, uniform or pointwise, of the local map regression estimator can thus be obtained.

Moreover, the upper bound in Theorem \ref{th:general} can be optimized by balancing the diameter of the cell $\mathcal{V}(x)$ with the empirical measure $n P_n^X(\mathcal{V}(x))$. This involves a fundamental trade-off between a small diameter and a large number of points. The challenge lies in effectively combining $P_n(\mathcal{V})$ and $\text{diam}(\mathcal{V})$ to obtain -- or not -- optimal rates of convergence. Observe that for a sufficiently large $n$, the empirical measure $P_n^X$ is close to the measure $P^X$. Furthermore, under the assumption of a density bounded from below, the measure $P^X(\mathcal{V}(x))$ is akin to the volume of the cell. Consequently, minimizing the overall error requires a geometric control of the cell's diameter relative to its volume. This relationship is precisely what motivates introducing the concept of shape regularity.

\section{Shape regularity}\label{sectionSR}

In this section, we show that anisotropic localizing sets slow the convergence rate, at least for regression functions with sufficient local variability. We then introduce the concept of shape regularity to control the geometry of the localizing sets.

\subsection{Leading example}\label{leading_ex}
 Consider the function $g:x\mapsto \sum_{k=1} ^d x_k $ defined on $[0,1]^d$. Set $d\geq 1$ and assume that $X \sim \mathcal U [0,1]^d$. Consider estimating $g$ at $0$ using a rectangular cell such that $\diam(\mathcal V)^d/ \lambda (\mathcal V) \geq \bar {\gamma}$ where $\bar {\gamma} >0$. Since $g$ is Lipschitz - note that each partial derivative of $g$ is actually \textit{equal} to one pointwise, \rev{so that $g$ varies at unit rate in every coordinate direction, making this configuration the hardest one for an anisotropic cell} - optimal rates are of order $ n^{ - 1/(d+2)}$. Next we show that, under standard conditions, the optimal rate cannot be achieved when $\bar \gamma $ grows with $n$. This is important as it means that the optimal rate cannot be attained except when $\bar{\gamma}$ is bounded, meaning that trees must have a certain shape regularity for being optimal.

\begin{proposition}\label{contre_ex}
Let $n\geq 1$ and $d\geq 1$. Suppose that \ref{cond:D} is fulfilled with $X \sim \mathcal U [0,1]^d$. Let $x \mapsto g(x) = \sum_{k=1} ^d x_k$ and suppose that the noise $\varepsilon$ satisfies $\EE(\varepsilon|X)=0$ and $\EE(\varepsilon^2 | X) \geq \sigma^2_{\text{min}}$.  Consider a local map $\mathcal V$ such that $ \mathcal V(0) = \prod_k[0,h_k]$, for some deterministic side lengths $h_k$. Let $\bar{\gamma}$ be such that $\diam(\mathcal V(0))^d/ \lambda (\mathcal V(0)) \geq \bar{\gamma}$. Whenever $2^{d+4} \log(2) \leq  n \prod_{k=1}^d h_k$, there exists a constant $C_d > 0$ depending only on $d$ such that $$   \EE [ ( \hat g_{\mathcal V} (0 ) -  g(0 ))^2 ]^{1/2}  \geq  C_d \left( \dfrac{\bar{\gamma} \sigma^2_{\text{min}}}{n} \right)^{1/(d+2)}. $$
\end{proposition}

%For instance, in the case $d=2$, if $h_1 = n^{-1/2}$ and $h_2 = n^{-1/6}$, then $n h_1 h_2 = n^{1/3}$. For all $n \geq \lfloor (2^6 \log(2))^3 \rfloor +1 = 7152$, the assumption of the preceding proposition is satisfied. Furthermore, our shape regularity factor $\bar{\gamma_n}$ diverges, indeed,$$\diam(\mathcal{V})^2/\lambda(\mathcal{V}) \geq h_2^2/(h_1 h_2) = h_2/h_1 = n^{1/3} = \bar{\gamma}_n\displaystyle \xrightarrow[n \to + \infty]{} +\infty.$$ Consequently, the minimax rate is not achieved.

More generally, the latter result still holds if $g(x) - g(0) \geq \sum_{k=1}^d x_k$ on $\mathcal V (0)$. An example of such function is, for instance, $g$ differentiable, $\nabla g (0)= (1,\dots,1)^T$ and $g$ convex. But many non-convex functions satisfy this condition, of course. Note also that the previous result can be extended to covariates $X$ having a density uniformly bounded from above and from below. Finally, by an easy conditioning argument, Proposition \ref{contre_ex}
still holds for side lengths $h_k$ that are independent of the sample, if $\bar \gamma$ is still deterministic. We stress that for most random trees, the randomness of the construction will actually require to consider a random shape parameter $\gamma$, and to study its stochastic variability.

\subsection{Minimal mass assumption}\label{4.1}

The next minimal mass assumption allows us to obtain an estimate for $P_n^X(\mathcal V(x)) $, which appears in the upper bound stated in Theorem \ref{th:general}.

\begin{enumerate}[label=(X), wide=0.5em,  leftmargin=*]
\item  \label{cond:density_X} 
%The distribution of the covariates vector $X$ admits a density function $f_X$ such that, f
For the local map $\mathcal V$ on $S_X$, there exists a function $\ell : x \mapsto \ell(x)> 0$ such that, almost surely, for all $x\in S_X$, 
$$ P^X(\mathcal V(x) )  \geq  \ell(x) \lambda (\mathcal V(x)  )  ,$$
where $\lambda$ stands for Lebesgue measure on $\mathbb R^d$.
\newcounter{nameOfYourChoice}
\setcounter{nameOfYourChoice}{\value{enumi}}
\end{enumerate}

Note that assumption \ref{cond:density_X} is easily satisfied when $X$ has a density $f_X$ bounded from below by a constant $b > 0$, by choosing $\ell(x) = b$ (see Section \ref{s_appli} for more precise example).
Moreover, note that the minimal mass assumption \ref{cond:density_X} is defined with respect to a specific local map $\mathcal{V}$ that we do not recall explicitely, and that will always refer in the following to the natural local map associated to the considered estimator. 

The minimal mass assumption is quite flexible, as it can be verified for the local maps arising from tree constructions. Indeed the  minimal mass assumption can be obtained by checking a more restrictive version involving some particular class of sets such as hyper-rectangles. We refer to Section \ref{s53} for more details.  %In the first case, $\mathcal V (x)$ is a (small enough) ball with positive (random) radius. As shown in \cite{jiang2019non}, such assumption is satisifed by bounded from below density on smooth sets $S_X$, as well as Gaussian or Laplace variables on $\mathbb R^d$ (\cite{gadat2016classification}). Some further details will be given in Section \ref{s61}. In the second example $\mathcal V (x) $ is an hyper-rectangle included in $S_X= [0,1]^d$ and therefore bounded from below densities on $S_X$ easily satisfy \ref{cond:density_X}.   

The following definition ensures that each element of the local map contains enough points. 

\begin{definition}
A VC local map  $x\mapsto  \mathcal V(x)$ with dimension $v>0$ is called $(\delta, n)$-large whenever, for all $x\in S_X$, almost surely,
    \begin{align*}
  n \max \left\{ P_n^X (\mathcal V(x))  ,     P^X(\mathcal V(x)) \right\}  \geq 36 \log\left(\dfrac{4 (2n+1) ^v }{\delta} \right).
\end{align*}
\end{definition}

Note that the latter inequality is easy to check in practice, as it suffices to make sure that enough data points are in each element of the local map.

\begin{theorem}
    \label{th2:general}
   Let $n\geq 1$ and $\delta \in (0,1/3) $.  Under \ref{cond:epsilon}, \ref{cond:D}, \ref{cond:reg4} and \ref{cond:density_X}, suppose that the local map is VC with dimension $v$ and is $(\delta, n)$-large, then we have with probability at least $1-3\delta$, for all $x\in S_X$,
   % \begin{align*}
   %    &|  \hat g_{\mathcal V}(x) - g(x)| \\
   %    &\leq \sqrt{\frac{ 2 \sigma^2 \log\left( \frac{ (n+1) ^v }%{\delta} \right)}{ n \PP_n(\mathcal V (x))  } } + \sup_{y \in %\mathcal V (x)} |g(y) - g(x)|
    %\end{align*}
\begin{align*}
        | \hat g_{\mathcal V}(x) - g(x)| \leq  \sqrt{\frac{ 3 \sigma^2}{  n  \ell(x) \lambda ( \mathcal V (x) )   }  \log\left( \frac{(n+1)^v}{\delta} \right)} + L(\mathcal V(x) ) \diam (\mathcal V(x) ).
    \end{align*}
\end{theorem}

The previous result differs from the one of Theorem \ref{th:general}, as the bound no longer depends on the number of data points in the associated local set, but instead on its Lebesgue volume. Together with the diameter, these two quantities will appear in the definition of the $\gamma$-shape regularity, so as to minimize the latter upper bound and therefore, to attain optimal rates of convergence for the underlying regression problem.

 As established in Theorem \ref{th2:general}, under the minimal mass assumption, the quantity $ |\hat g_{\mathcal V}(x) - g (x) |$ is bounded by
 $  \sqrt { 1 / ( n \lambda (\mathcal V (x)) )    }   +   \diam (\mathcal V (x))  $, up to constants and log terms. Theorem \ref{th2:general} thus allows us to understand that a trade-off between the volume and the diameter must be achieved to reach optimal rates.
In this regard, first note that the volume cannot be greater than the diameter to the power $d$, as we always have $ \lambda (\mathcal V(x) )  \leq \diam ( \mathcal V(x) )^d $. Incorporating that constraint when optimizing the previous bound leads to $\lambda (\mathcal V(x) ) = \diam ( \mathcal V(x) )^d = n^{-d/(d+2)}$, which is the optimal rate in our regression problem. In contrast, if $ \diam ( \mathcal V(x) )^d = \gamma_n \lambda (\mathcal V(x) )   $ with $\gamma_n \to \infty$, then the bound of Theorem \ref{th2:general} gives a slower, suboptimal convergence rate. This reasoning motivates the introduction of the shape-regularity condition, in the next section.

\subsection{Isotropic requirements through shape regularity}\label{sec_SR}
%A key concept is now introduced, which will help characterize the convergence rates of the local map regression estimators. As established in Theorem \ref{th2:general}, under the minimal mass assumption, the quantity $ |\hat g_{\mathcal V}(x) - g (x) |$ is bounded by
% $  \sqrt { 1 / ( n \lambda (\mathcal V (x)) )    }   +   \diam (\mathcal V (x))  $, up to constants and log terms. Theorem \ref{th2:general} thus allows us to understand that a trade-off between the volume and the diameter must be achieved to reach optimal rates.
 %or the well known AM-GM inequality $\lambda (\mathcal V(x) )^{1/d}  \leq (1/d) \diam_1 ( \mathcal V(x) )  $). 
%Incorporating that constraint when optimizing the previous bound 
%with respect  to $ \lambda (\mathcal V(x) )$ and $\diam ( \mathcal V(x) )$  under $constraint that $\lambda (\mathcal V(x) )  \leq \diam ( \mathcal V(x) )^d $
%leads to $\lambda (\mathcal V(x) ) = \diam ( \mathcal V(x) )^d = n^{-d/(d+2)}$, which is the optimal rate in our regression problem. 

We have shown that if $\diam(\mathcal{V}(x))^d \geq \bar{\gamma}_n\lambda(\mathcal{V}(x))$ with $\bar{\gamma}_n\to \infty$ for a given cell $\mathcal{V}(x)$, then the bound provided by Proposition \ref{contre_ex} yields a slower, suboptimal convergence rate. This reasoning motivates the introduction of the following notion of shape-regularity.

%The diameter derived from the $\ell_2$-norm will be denoted simply as the diameter.

\begin{definition}\label{def:gamma_regular}
For $\gamma>0 $, a set $ V $ is called $\gamma$-shape-regular ($\gamma$-SR) if
$\diam(V)^d \leq \gamma \lambda (V)$.
\end{definition}

\noindent The previous condition can be interpreted as a volume condition: the volume of $V$ should be of the same order as the volume of the smallest ball containing $V$. In this regard, first recall, as already noted in the previous section, that the volume cannot be greater than the diameter to the power $d$, as we always have $ \lambda (\mathcal V(x) ) \leq \diam ( \mathcal V(x) )^d$. Roughly speaking, the shape of $V$ is not that different from that of a ball or a hypercube, i.e. $V$ is ``almost isotropic''. Moreover, it does not depend on the covariates density, making it easy to check in practice. 
 
We provide now an alternative to Definition \ref{def:gamma_regular}, specifically designed for local maps valued in the set of hyper-rectangles. %Its advantage is that it can be easily imposed during tree construction. 
For any hyper-rectangle $A\subset  S_X$, let $h_-(A)$ and $h_+(A)$ denote the smallest and largest side length, respectively.
\begin{definition}\label{def:beta_regular}
For $\beta>0 $, a hyper-rectangle $ A$ is called $\beta$-shape-regular ($\beta$-SR) if 
$ h_+ (A)  \leq \beta h_-(A)$.
\end{definition}
\noindent It is easily seen that when a set $V$ is an hyper-rectangle, the $\gamma$-SR property is related to $\beta$-SR. This is the subject of the following proposition. 
\begin{proposition}\label{link_beta_gamma}
     A $\beta$-SR hyper-rectangle  is $\gamma$-SR with $\gamma = \beta^d \, d^{d/2}$. Conversely, a $\gamma$-SR hyper-rectangle  is $\beta$-SR with $\beta= \gamma$.
\end{proposition}
\noindent The two definitions of shape regularity, $\gamma$ and $\beta$, are therefore equivalent in the case of hyper-rectangles. More precisely, the first implication in Proposition \ref{link_beta_gamma} will be of particular interest for us, as it will allow us to show that some regression trees are $\gamma$-shape-regular. In practice, one way to obtain a $\beta$-SR (and therefore $\gamma$-SR) tree is to allow only for $\beta$-SR splits when growing the tree, i.e., valid splits in light of Definition \ref{def:beta_regular}. This is easily imposed, as it only requires one to restrict the  optimization domain when finding the optimal split. We further develop this aspect in Section \ref{s53}.

Note that, in dimension $d=1$, trees are necessarily shape-regular for $\gamma=\beta=1$ as $ h_- = h_+$. From this perspective, dimension $1$ plays a special role and might exhibit convergence properties that would not generalize to larger dimensions.

%\subsection{Minimal mass assumption and shape regularity of local maps}\label{4.1}

Let us now introduce the following definition, which requires that all elements of the local map are $\gamma $-SR. 

\begin{definition}
A local map $x\mapsto  \mathcal V (x)$ is $\gamma$-SR  if all elements in \hbox{$\{\mathcal V (x) : x\in S_X\}$} are $\gamma$-SR. 
\end{definition}

To validate the $\gamma$-SR condition, we now provide some error rates for such $\gamma$-SR local maps, when choosing a suitable value for the volume. In the next statement, we use the notation
$ f \lesssim g $ when there exists a universal constant $ a >0$ such that 
$   f  \leq a  g .$
We write $ f \asymp g $ whenever $f\lesssim g $ and $g\lesssim f$, 

\begin{theorem}\label{main_result_localizing_map}   
Under the assumptions of Theorem \ref{th2:general}, if the local map is $\gamma$-SR and if for all $x\in S_X$,
$\lambda(\mathcal V(x) ) \asymp ({\log ( {(n+1)^v} / { \delta} ) } / { n} )  ^{d/(d+2)}$, we have, with probability at least $1 - 3\delta $, for all $x\in S_X$,
    \begin{align*}
        & | \hat g_{\mathcal V}(x) - g(x)| \lesssim c(x) \left(\frac{1 }{ n}\log\left( \frac{(n+1)^v}{ \delta} \right) \right) ^{1/(d+2)}
    \end{align*}
     where $c(x) = \sqrt{{ \, \sigma^2}/\ell(x)} +  L(\mathcal V(x) )  \gamma^{1/d}.$ 
 %In addition, whenever $S_X$ is bounded and $ \ell(x) \geq b >0 $ for all $x\in S_X$, we have, with probability at least $1 - 3\delta $,
 % \begin{align*}
 %        & \sup_{x\in S_X} | \hat g_{\mathcal V}(x) - g(x)| \lesssim c \left(\frac{\log\left( \frac{(n+1)^v}{ \delta} \right) }{ n} \right) ^{1/(d+2)}
 %    \end{align*}
 %    where $c = \sqrt{{3 \, \sigma^2}/b} +  L  \gamma^{1/d}.$
    \end{theorem}

% We prove in Theorem \ref{main_result_localizing_map} some probability bounds for the pointwise and sup-norm errors. 
Note that our pointwise probability bound is valid \textit{for all} $x$ in the domain $S_X$, but with a pre-factor $c(x)$, introduced in the minimal mass assumption \ref{cond:density_X}. When this pre-factor can be bounded from below uniformly in $x$ (see the examples of Section \ref{s_appli}), the above is in turn a uniform error bound.

The requirement about the order of the volume $\lambda(\mathcal V(x)) $ in Theorem \ref{main_result_localizing_map} allows to minimize the bound in Theorem \ref{th2:general}. In most practical situations, this precise parameter cannot be directly tuned and one is only able to select another hyper-parameter that will in turn impact the value of $\lambda(\mathcal V(x))$, as observed in the examples of the next section. However, it serves to illustrate the potential rate of convergence achievable under our assumptions.

Note also that our choice for the order of the volume $\lambda(\mathcal V(x))$ depends on the confidence level $\delta$, thus making the estimator $\delta$-dependent. An alternative choice, such as $\lambda(\mathcal V(x)) \asymp (\log(n) / n)^{d/(d+2)}$, has the advantage of being independent of $\delta$.
Such a choice allows us to extend our result to pointwise and uniform convergence rates in expectation of order $(\log(n)/n)^{1/(d+2)}$, which corresponds to the minimax rate in expectation for the sup-norm error (see, for instance, \citep{tsybakov2009}). 
%In comparison, the minimax error rate in expectation is of order $n^{-1/(d+2)}$ for pointwise estimation. %Thus, on the one hand, we only recover the minimax rates in expectation up to some extra-log factors, but on the other hand, our bounds are more informative than criteria in mean, since we describe the behavior of the whole tails of the errors.
%where the the constants $c$ may differ from the respective ones of Theorem \ref{main_result_localizing_map}, due to the integration step. As shown for instance in \cite{tsybakov2009}, the above sup-norm rate is minimax optimal for the estimation of Lipschitz functions, whereas the pointwise rate actually contains an extra-log factor. This factor is inherent to our VC-type approach for deriving deviation bounds, and the question whether it could be removed in general for $\gamma$-SR local maps remains open.
%It is worth noticing that if the density $f_X$ of $X$ is uniformly bounded from below over $S_X$ and if $g$ is $L-$Lipschitz on $S_X$, then Theorem \ref{main_result_localizing_map} gives a convergence rate that actually holds in sup-norm on $S_X$. %We also point out that this result is more theoretical than practical as setting $\lambda(\mathcal V(x) ) $ in such a way might not be possible in practice.

\section{Applications}\label{s_appli}

In what follows, the pivotal role of shape regularity is illustrated through several examples. First, the classic $k$-NN regression method is studied. Second, we examine the shape regularity of random split trees and identify a specific subclass that exhibits suboptimal convergence rates. Finally, we discuss the application of our results to the CART algorithm and also propose an amended version that incorporates shape regularity constraints.

\subsection{Revisiting the nearest neighbor method}\label{sec_k_nn}

%Optimal convergence of shape-regular local map estimators}\label{sec_SR_opt}

%The geometric analysis of Wager trees reveals inherent difficulties in achieving shape ragularity and therefore optimal rates (see Sections \ref{wager_sousopt} and \ref{leading_ex}). %In this section, we transition from the analysis of geometric drift to the study of shape-regular partitions. 
%In this section, we demonstrate that shape regularity is not merely a desirable property but a sufficient condition for statistical optimality. To conclude this section, we apply our result to the case of $k$-nearest neighbors. This provides a natural illustration of our theoretical framework, since $k$-nearest neighbors algorithms are inherently based on shape-regular local maps.

%\subsubsection{An application to $k$-nearest neighbors}
In this section, we apply our result to the case of $k$-nearest neighbors. This provides a natural illustration of our theoretical framework, since $k$-nearest neighbors algorithms are inherently based on shape-regular local maps. This is an application of Theorem \ref{th:general} applied to $k$-NN, where the local map $\mathcal{V}$ is a ball. Indeed, in Theorem \ref{th:general}, the variance term features $P_n^X(\mathcal{V}(x))$, which is exactly $k/n$. Furthermore, observe that the nearest neighbors algorithm is based on a shape-regular local map, since we have the relationship $$\diam(\mathcal{V}(x))^d = (2 \hat \tau _{n,k}(x))^d = \dfrac{2^d}{V_d} \lambda(\mathcal{V}(x))$$
where $V_d$ denotes the volume of the unit ball in $\mathbb{R}^d$ and $\lambda(\mathcal{V}(x))$ is the Lebesgue measure of the cell. This relation highlights that for $k$-NN, the diameter is intrinsically tied to the volume, ensuring a perfect shape regularity. Consequently, this fits within the arguments developed in Section \ref{4.1} regarding the shape regularity of local maps.

More precisely, nearest neighbors regression estimators are local maps estimators for which $\mathcal V (x) = B (x, \hat \tau _{n,k}(x) ) $ where $ \hat \tau _{n,k}(x)$ has been defined in Section \ref{s2}, Example \ref{ex3}. In contrast with the general approach developed in the previous section, which relies on Assumption \ref{cond:density_X}, we here no longer consider the (possibly random) local map $\mathcal V$ but rather focus on a given class of balls (with small enough radius).

\begin{enumerate}[label=(XNN), wide=0.5em,  leftmargin=*]
\item  \label{cond:density_XNN} %The covariates vector $X$ admit a positive density function $f_X$ on $S_X$ and 
There is a positive function $\ell $ defined on $S_X$ and $T_0>0$ such that, for all $x\in S_X$ and $\tau \in (0, T_0)$,
$$ P^X( B(x, \tau) )  \geq   \ell(x)\tau ^d.$$
\end{enumerate}

  As we will see below, Assumption \ref{cond:density_XNN} is sufficient when dealing with nearest neighbors regression estimators. Moreover, Assumption \ref{cond:density_XNN} is satisfied whenever $X$ has a density $f_X$ which is bounded below by a constant $b> 0$ on $S_X$ (in which case $S_X$ must be bounded) and when $S_X$ satisfies $ \int _{S_X \cap B(x,\tau )} d\lambda \geq \kappa_0 \int _{B(x,\tau )} d\lambda$, % \geq \kappa_0 \tau ^d V_d$, where $V_d = \lambda (B(0,1)) $, 
  for all $\tau \in (0, T_0)$. Assumption \ref{cond:density_XNN} can also be satisfied when $S_X$ is unbounded. Several examples are given in  \cite{gadat2016classification}. 

Following an approach quite similar to the proof of Theorem \ref{th:general}, we obtain the following result.

\begin{theorem}\label{th:NN}
Let $\delta \in (0,1/3) $, $n\geq 1$, $d\geq 1$ and $k\geq  8    \log( 4 (2n+1) ^{(d+1) } / \delta ) $.  Let $\mathcal V $ be obtained from nearest neighbors algorithm, as detailed in Example \ref{ex3}. Suppose that \ref{cond:epsilon}, \ref{cond:D}, \ref{cond:reg4} and  \ref{cond:density_XNN} are fulfilled. We have, with probability at least $1-3\delta$, for all $x\in S_X$ such that $2 k \leq   T_0^d n \ell(x)$,
\begin{align*}
    &|\hat g_{\mathcal V} (x) - g(x) | \leq  \sqrt{\frac{2\sigma^2 }{ k }\log \left( \frac{(n+1)^{d+1}}{\delta }\right)  } + 2 \left(\frac{ 2 k   }{  n \ell(x)  }   \right)^{1/d} L(\mathcal{V}(x)).
\end{align*}
\end{theorem}

Note that the conditions on the value of $k$ are satisfied for $n$ sufficiently large   and $k\asymp n^a $, for any $a\in (0,1)$. 
To our knowledge, the above result is new among the nearest neighbors literature, in which uniform deviation inequalities are provided, but only for densities uniformly bounded away from $0$. Such results have been investigated recently in \cite{jiang2019non} and \cite{portier2021nearest} for compactly supported covariates. In contrast, the above upper bound is valid for all $x$ in any domain $S_X$, at the price of accounting for regions with low density values, which may in general deteriorate the accuracy locally. We have the following corollary, in which we consider an optimal choice for $k$, as well as a uniform lower bound on the density.

\begin{corollary}\label{corknn}
In Theorem \ref{th:NN}, assuming that $ n $ is sufficiently large and choosing the integer $k \asymp n^{2/(d+2)} \log((n+1)^{d+1}/\delta)^{d/(d+2)}$, yields the following inequality, with probability at least $1-3\delta$ and for all $x\in S_X$, 
   \begin{align*}
      |\hat g_{\mathcal V}(x) - g(x) |  \lesssim c(x) \left( \frac{1}{n} \log \left( \frac{(n+1)^{d+1}}{\delta }\right) \right)^{1/(d+2)},
   \end{align*}
    where $c(x) =   \sqrt{2\sigma^2} + 2 L(\mathcal{V}(x)) \left({2}/\ell(x) \right)^{1/d}.$  
 %    Moreover, if $\ell(x) \geq b >0 $, for all $x\in S_X$,  we have, when $n$ is sufficiently large, with probability at least $1 - 3\delta $,
 % \begin{align*}
 %        & \sup_{x\in S_X} | \hat g_{\mathcal V}(x) - g(x)| \lesssim c \left( \frac{1}{n} \log \left( \frac{(n+1)^{d+1}}{\delta }\right) \right)^{1/(d+2)},
 %    \end{align*}
 %   where $c =   \sqrt{2\sigma^2} + 2 L \left(2/b \right)^{1/d}$.
\end{corollary}

Note that the convergence rate is the same as in the abstract Theorem \ref{main_result_localizing_map}. However, the constant $c$ in the first above statement differs significantly from that of Theorem  \ref{main_result_localizing_map} as when $\ell(x)$ is small, the constant in Theorem \ref{main_result_localizing_map} is of order \( { \ell(x) } ^{-1/2} \), while in Corollary \ref{corknn}, it scales as \(  \ell(x) ^{-1/d} \). This is explained by the fact that, in the proofs of the respective results, the value of $\ell(x)$ has an effect on the variance term, that contributes to the bound in Theorem \ref{main_result_localizing_map}, while it appears in the bias term for Corollary \ref{corknn}. 

\subsection{On random split trees: from  consistency to non-optimal rates}\label{sec_wag}

\subsubsection{Background}

We investigate the class of random split trees as considered in recent studies including \cite{meinshausen2006quantile,biau2012analysis,wager2018estimation}. For clarity, we adopt the formalism developed in \cite{wager2018estimation}, %These trees are particularly suitable to estimate average treatment effect and to provide some inferential guarantees.
%they rely on a structural regularity criterion -- specifically, a minimum proportion of observations in each child node -- which shares a common goal with shape regularity: the prevention of cell degeneracy. Some links will be drawn with our SR condition.
%Let us formally define the Wager class of trees. 
where the tree construction follows from four key properties: honest, symmetric, random-split, and $\alpha$-regular.
% \begin{definition}
%     A Wager tree is honest, symmetric, random-split, and $\alpha$-regular. 
% \end{definition}
We now state each assumption with a brief informative description. 

\noindent \textbf{Honesty.} The {honesty} assumption captures the independence between leaf construction and prediction. It is enforced in one of two ways: either the leaf construction relies solely on the covariates, or the data is split into two independent subsamples — one used to build the leaves (using both $X$
 and $Y$) and one used for prediction. Honesty allows one to work conditionally on the subsample used to build the leaves. In our work, the uniform result (over classes of sets) established in Theorem \ref{th:general}, make the honesty assumption unnecessary. Therefore, it will not be considered in what follows.

\noindent \textbf{Symmetry.} The symmetry property asserts that  the predictor's output does not depend on the indexing order of the training examples. It is of limited importance as it is satisfied by \rev{most practical algorithms, such as standard implementations of CART}. This will not be needed for our development either.

% Furthermore, the symmetry condition ensures that the predictor's output does not depend on the indexing order of the training examples. While this is a fundamental statistical consistency requirement for the theoretical study of trees and forests (\cite{wager2014asymptotic}), it does not necessarily enhance practical performance, 

 \noindent \textbf{Random-split.} A tree satisfies the {random-split} property if, at every step, the probability that the next split occurs along the $j$-th feature is bounded below by $\pi/d$ for some $0 < \pi \leq 1$. Moreover, it is also required that the sequence of selected features for splitting is independent. This ensures that no variable is indefinitely ignored and that all directions are eventually explored, providing comprehensive coverage of the feature space. 

 \noindent \textbf{$\alpha$-regularity.} The {$\alpha$-regularity} requires each split to keep at least a fraction $\alpha \in (0,1/2]$ of the available training examples in each child node. This prevents degenerate leaves and controls the effective depth of the tree, ensuring each leaf retains enough data to compute a reliable local mean. 
 %While, as we will see below, this facilitates convergence by balancing bias and variance, it however lacks in pratice the ability of CART to create leaves that adapt to complex regions, potentially resulting in a loss of local precision.  

The geometric integrity of the cells of the above trees rests on the random-split and $\alpha$-regularity hypotheses.
Intuitively, the random split condition ensures that all splitting directions are considered at each node, while $\alpha$-regularity prevents splits at extreme quantiles. For uniformly distributed covariates, this implies that the constructed leaves resemble squares rather than elongated rectangles, which supports the intuition that shape regularity should be satisfied. We revisit results from \cite{biau2012analysis,wager2018estimation,duroux2018impact} and confirm that this intuition is partially correct: consistency is indeed achieved for these trees. The main reason is that random split condition and
$\alpha$-regularity implies that the diameter of the leaves \rev{goes} to $0$. However, and perhaps surprisingly, the optimal rate cannot be achieved. We establish that shape regularity fails with positive probability for random split trees, which constitutes a negative result and calls for an alternative construction. 

%After showing that shape regularity leads to minimax-optimal rates, we propose such an alternative within the CART framework.

%In the following, we evaluate how the interplay between this $\alpha$-regularity and the random-split mechanism impacts the estimator's minimax optimality.

\subsubsection{Consistency of $\alpha$-regular and random split trees}\label{sec_wager}

The aim is to obtain a non-asymptotic upper bound on the error of $\alpha$-regular and random split tree.
%In this section, we prove that Wager trees are consistent for choices of depth $N$ of order of $\log(n)$, where $n$ denotes the number of data points $(X_i,Y_i)_{i=1,\ldots,n}$. 
As before, we consider a collection $(X,Y), (X_i,Y_i)_{i=1,\dots,n}$ of independent and identically distributed random variables with law $P$ and such that $X$ has a density on $ [0,1]^d$. We introduce a variant of the minimal mass assumption \ref{cond:density_X} stated in Section \ref{4.1}. %Define $R(x,h)$ as the hyper-rectangle with center $x$ and size length $h = (h_1,\ldots, h_d)^T \in \mathbb R^d $.

\begin{enumerate}[label=(XTREE), wide=0.5em,  leftmargin=*]
\item  \label{cond:density_XCART} The random variable $X$ admits a density function $f_X$ on $[0,1]^d$. There are two constants $M, b > 0$ such that $\forall x \in \mathcal [0,1]^d$,  $b \leq f_X(x) \leq M$.
\end{enumerate}

Note that the above assumption is stronger but more convenient than \ref{cond:density_X}, as it no longer involves the local map $\mathcal V$, that depends on the sample.  The tree structure is defined with the help of the following notation. For any $x \in [0,1]^d$, consider a sequence of nested cells $(\mathcal{V}_i(x))_{0 \leq i \leq N}$ such that $x \in \mathcal{V}_N(x)\subset \mathcal{V}_{N-1}(x)\subset \dots \mathcal{V}_{0}(x) = [0,1]^d$. Each transition from one parent $\mathcal{V}_{i-1} (x)$ to its children $\mathcal{V}_{i}(x)$ is associated with a cutting direction $D_i(x)$ and an axis-aligned split position $U_i(x)$.
The following condition formally requires the \textit{random split} and the $\alpha$-\textit{regularity} condition as introduced previously. 
%For the sake of simplicity, we shall omit the explicit dependence on $x$ and denote them simply as $U_i$ and $D_i$.
\begin{enumerate}[label=(W), wide=0.5em,  leftmargin=*]
\item \label{cond:Wager} For any $x\in [0,1]^d$, the directions $(D_i(x))_{i=1,\dots,N}$ are independent and for every $j \in \{1, \dots, d\}$, $\mathbb{P}(D_i(x)=j) \geq \pi/d$ with $\pi\in (0 ,  1]$. Moreover, there is $\alpha \in (0,1/3]$ such that, for all $x\in [0,1]^d$ and $i \in \{1,\ldots, N\}, $%for every internal node $\mathcal{V}$, its two children $\mathcal{V}_L$ and $\mathcal{V}_R$ satisfy
%$$\min \left( {P}_n^X(\mathcal{V}_L), {P}_n^X(\mathcal{V}_R) \right) \geq \alpha {P}_n^X(\mathcal{V}),$$ 
 $$  P _n^X(\mathcal{V}_{i}(x) )  \geq \alpha P _n^X(\mathcal{V}_{i-1}(x) ) ,$$
where we recall that $P_n^X(V) = n^{-1} \sum_{i=1} ^n \ind _V(X_i)$ is the proportion of observations $X_i$ falling into a cell $V$. 
\end{enumerate}

Note that the split sizes $(U_i(x))_{i=1,\dots,N}$ may depend on both the $(D_i(x))_{i=1,\dots,N}$ and the data $(X_i,Y_i)_{i=1,\dots,n}$. This ensures that the above construction is admissible. Specifically, suppose there are at least $2$ points in  $\mathcal{V}_N(x) $. %It follows that
Since $X$ possesses a density, they cannot be axis-aligned and therefore, regardless of the cutting direction $ D_{N+1}(x)$, preserving \ref{cond:Wager} is possible by choosing $U_{N+1}(x)$ so to split the points into two halves ($\alpha \leq 1/3$ is to deal with the situation where an odd number of points, e.g., $3$, are in the cell). But since each split preserves at least a fraction $\alpha$ of the points from the parent node, it holds that $P_n^X(\mathcal{V}_N(x)) \geq \alpha^N$. %and consequently, $P_n^X(\mathcal{V}_i(x)) \geq \alpha^N$ for all $i= 0,\ldots, N$. 
Hence, whenever $n \alpha^{N}>1 $, there are at least $2$ points in  $\mathcal{V}_N(x) $. As a consequence, whenever $n\alpha ^N>1$, the tree can be grown further while satisfying \ref{cond:Wager}.

%Indeed, as long as there are at least two points in the cell, it is always possible to continue splitting for a certain $\alpha$.

%the following condition $8 \log ( 4 (2n+1)^{2d} / \delta) \leq n \alpha^{N+1}$ ensures that the number of points within each cell -- given by $n P_n(\mathcal{V}_N(x))$ -- is at least of the order of $\log(n)$. This construction is feasible since the $X_i$ possess a density, preventing them from being concentrated on degenerate sets (such as being axis-aligned). Indeed, as long as there are at least two points in the cell, it is always possible to continue splitting for a certain $\alpha$.

The following proposition is key to obtain the consistency as it implies that the diameter of any cell shrinks to $0$. This is a relevant property in light of Theorem \ref{th:general} where one term in the upper bound is proportional to the diameter of the cell. Define 
$$ h_+(\mathcal{V}_N(x)) = \max_{j=1, \dots, d} h_j(\mathcal{V}_N(x))  ,$$
where $h_j(\mathcal{V}_N(x))$ denotes the length of the $j$-th side of the cell $\mathcal{V}_N(x)$.

%\color{blue}(This is only when $X$ has a density. In this case, the tree can be grown until the cell contains just $1$ points. We can then stop the growing process. This is in fact not really needed because we will assume that $N$ is small enough to keep a sufficient number of points. Is it correct? Perhaps we could say that clearly... I think that this assumption 8 $\log ( 4 (2n+1)^{2d} / \delta) \leq n \alpha^{N+1}$ is in fact saying that we have at least an order of $\log(n)$ points in the cell)\color{black}. \color{black}
%A condition, that is not  \ref{cond:Wager} is weaker in some respects, it can be viewed as a subclass of Wager's framework, since we require the $D_i$ to be independent.

%According to condition \ref{cond:Wager}, the total number of splits performed along the $j$-th coordinate on a given path is defined by the random variable $N_j := \sum_{i=1}^N \ind_{D_i=j}$. %Geometrically, starting from an initial cell with side lengths equal to 1, we denote $h_j$ the final side length in direction $j$ after $N_j$ splits. 
%The $\alpha$-regularity condition ensures that each split maintains at least a fraction $\alpha$ of points between the parent and child cells. Finally, we assume that the density of the observations $(X_i)_{i=1,\dots,n}$ is bounded from below and above by constants $b$ and $M$, respectively which ensures that the maximum side lenth of the consider cell decreses exponentially fast to 0 with the number of splits $N$.

\begin{proposition}\label{propWager}
Let $n \geq 1$, $N\geq 1$ and $\delta \in (0,1/3) $. Grant Assumptions \ref{cond:D}, \ref{cond:density_XCART}, \ref{cond:Wager}. Suppose that $N \pi \geq 8d \log(d/\delta)$ and $16 \log ( 4 (2n+1)^{2d} / \delta) \leq n \alpha^{N}$. Then, for any $x\in [0,1]^d$, the maximum side length $h_+(\mathcal{V}_N(x))$ satisfies, with probability at least $1-3\delta$, for all $x \in [0,1]^d$,
$$h_+(\mathcal{V}_N(x))  \leq \left(1-\dfrac{\alpha b}{8M}\right)^{{N\pi}/{(2d)}}.$$
\end{proposition}

Note that 
%the density of the observations $(X_i)_{i=1,\dots,n}$ is bounded from below and above by constants $b$ and $M$, respectively which ensures that the maximum side lenth of the consider cell decreses exponentially fast to 0 with the number of splits $N$. 
%In proposition \ref{propWager}, 
two conditions are imposed on the number of splits $N$. The first condition requires $N$ to be sufficiently large, so that each direction is split many times. The second requires $N$ to be small enough to maintain a critical mass of data points within each cell. 

Furthermore, %in accordance with the $\alpha$-regularity criterion, the empirical measure of a cell $\mathcal{V}_N$ at the $N$-th iteration, denoted as $P_n^X(\mathcal{V}_N(x))\geq n\alpha^N$, is bounded below by $\alpha^N$. This inequality reflects the cumulative effect of the recursive partitioning: since each split preserves at least a fraction $\alpha$ of the points from the parent node, a leaf at depth $N$ necessarily contains at least an $\alpha^N$ proportion of the total sample. 
%Consequently, 
 applying Theorem \ref{th:general} and using that $P_n^X(\mathcal{V}_N(x))\geq \alpha^N$, we obtain an upper bound on the point-wise error of $\alpha$-regular and random split trees.

\begin{theorem}\label{Wager}
Let $n \geq 1$, $N\geq 1$ and $\delta \in (0,1/5) $. Grant Assumptions \ref{cond:D}, \ref{cond:density_XCART}, \ref{cond:epsilon}, \ref{cond:reg4} and \ref{cond:Wager}. Suppose that $N \pi \geq 8d \log(d/\delta)$ and $16 \log ( 4 (2n+1)^{2d} / \delta) \leq n \alpha^{N}$. Then it holds,  with probability at least $1 - 5 \delta $, for any $x\in [0,1]^d$,
   % \begin{align*}
   %    &|  \hat g_{\mathcal V}(x) - g(x)| \\
   %    &\leq \sqrt{\frac{ 2 \sigma^2 \log\left( \frac{ (n+1) ^v }%{\delta} \right)}{ n \PP_n(\mathcal V (x))  } } + \sup_{y \in %\mathcal V (x)} |g(y) - g(x)|
    %\end{align*}
        \begin{align*}
       |  \hat g_{\mathcal V}(x) - g(x)| \leq \sqrt{\frac{ 2 \sigma^2 }{ n \alpha^N  } \log\left( \frac{ (n+1) ^{2d}}{\delta} \right)} + L(\mathcal V(x)  ) \sqrt{d} {\left(1-\dfrac{\alpha b}{8M}\right)^{N\pi/(2d)}}.
    \end{align*}
\end{theorem}

As soon as $N \to \infty$ and $\log(n)/(n \alpha^N) \to 0$, we obtain the consistency of the partition-based estimator. This holds, in particular, for the choice of depth $N$ such that $N = (-2\log(\alpha))^{-1} \, \log(n)$. With this choice of $N$, and for a fixed $\delta \in (0,1/6)$, the conditions $N \pi \geq 8d \log(d/\delta)$ and $16 \log ( 4 (2n+1)^{2d} / \delta) \leq n \alpha^{N} = \sqrt{n} $ are indeed satisfied for sufficiently large $n$. Moreover, by optimizing the bound, we obtain that $N$  should be set as $\log(n)/C$ where $C = -\log(\alpha) - \log(1- \alpha b / (8M)) \pi/d$.
It follows that the total error scales as $n^{-s}$ with $s = {(2 + dK)^{-1}}$ where $K = {2 \log(\alpha)}/{(\pi \log\left( 1 - {\alpha b}/{(8M)} \right))} > 52$\rev{, the numerical lower bound following from the constraints $\alpha \leq 1/3$, $\pi \leq 1$ and $b \leq M$}. Note that in the classical minimax framework, the variance typically scales as $1/\sqrt{nh}$ while the bias scales as $h^{1/d}$. Although our variance follows the $1/\sqrt{nh}$ rate for $h=\alpha^N$, our setting differs because the bias is of order $h^{c/d}$ for $c = \pi \log(1-\alpha b/(8M)) / (2\log(\alpha))$ and this consequently yields a sub-optimal bound. 

It is worth noticing that our previous analysis, based on condition \ref{cond:Wager}, does not require the original symmetry and honesty conditions of \cite{wager2018estimation}. %The next two results will be thus valid for the corresponding subclass of Wager trees. 

The previous results may be viewed as ``uniform'' version of the point-wise results obtained in \cite{wager2018estimation,duroux2018impact}. Compared with \cite{wager2018estimation}, we obtain a similar convergence rate, although it is expressed in a different form since our trees are not fully grown and our result holds for any depth $N$. In contrast to \cite{duroux2018impact}, our bound is derived under a different splitting rule based on the sample median, which leads to a different rate. It is also worth noting that our result provides a deviation bound, whereas \cite{duroux2018impact} establishes an $L_2$-bound. Furthermore, our statement is fully explicit, both in terms of the problem-dependent constants and the confidence parameter $\delta$.
%$$N \approx \frac{-\log(n)}{ \log(\alpha)+ \log(1- \alpha b / (12M)) \pi/d  }.$$

\subsubsection{On the insufficiency of random split for shape regularity}\label{wager_sousopt}

%The geometric instability demonstrated in Section \ref{sec_sub} for general random partitions finds a direct illustration in the case of certain trees satisfying the condition of $\alpha$-regularity and the possibility of splits along all directions, which are fundamental criteria for Wager trees \cite{wager2014asymptotic}. This model, introduced in Section \ref{sec_wager}, primarily relies on a simple stochastic splitting rule. 

The aim here is to show that the Wager conditions, random split and $\alpha$-regularity are not sufficient to ensure shape regularity of the cells. 

\rev{To this end, we isolate the sub-class of random-split trees in which the cutting directions are, in addition, independent of the split positions; we call these ``blind'' trees.} We study here the class of ``blind'' trees characterized by a strong independence structure: the cutting directions are uniformly distributed and independent from the split positions. This is formally stated in the following assumption.

\begin{enumerate}[label=(BL), wide=0.5em,  leftmargin=*]
\item \label{cond:indep_split_position} For all $x\in [0,1]^d$, the cutting directions $(D_i(x))_{i=1,\dots,N}$ and the split positions $(U_i(x))_{i=1,\dots,N}$ are independent. $(D_i(x))_{i=1,\dots,N}$ is an independent collection of random variables with common uniform distribution over $\{1, \dots, d\}$. $(U_i(x))_{i=1,\dots,N}$ are away from the edges i.e., there exists $\rho \in (0,1/2]$ such that for all $i \in \{ 1,\dots,N\},  U_i(x) \in  [\rho, 1-\rho]$. 
\end{enumerate}

%Consider a tree of depth $N$ where the cutting directions $(D_i)_{i=1,\dots,N}$ are i.i.d., uniformly distributed over $\{1, \dots, d\}$ and independant of the relative split positions $(U_i)_{i=1,\dots,N}$. Assume that there exists $\rho >0$ such that for all $i \in \{ 1,\dots,N\},  U_i \in  [\rho, 1-\rho]$. Furthermore, suppose that the density $f_X$ of the observations is such that there exist $m, M > 0$ with $m \leq f_X(x) \leq M$ for all $x \in S_X$. For all $N \geq 1$ and $n \geq 1$ such that $8 \log(16(2n+1)^{2d}) \leq n {\alpha}^{N+1}$, the following inequality holds \begin{align*}
%\PP\left(\frac{h_+(\mathcal V_N (x) )}{h_-(\mathcal V_N (x) )}\geq \exp\left(\sqrt{\frac{N  \log(1-\tilde \alpha)^2}{d}}\right)\right) &\geq \dfrac{ \log(1-\tilde \alpha)^4}{48d \log(\tilde \alpha)^4}
%\end{align*}
%where $\mathcal{V}_N(x)$ denotes the cell containing $x$ at depth $N$ and $\tilde{\alpha} = \alpha m / (12M) > 0$. Consequently, these Wager trees are not shape regular with strictly positive probability.
\begin{theorem}\label{Wager_subopt}
Consider a tree of depth $N \geq 1$ and grant \rev{Assumptions} \ref{cond:density_XCART} and \ref{cond:indep_split_position}. For all $x\in  [0,1]^d$, and all $N \geq 1$ and $n \geq 1$ such that $16\log( 192d \nu^4 (2n+1)^{2d}) \leq n (b\rho / M ) ^{N} $,
with $\nu = \log(\rho)/\log(1-\rho)$, the following holds with probability at least $\nu^4/(24d)$: 
\begin{itemize}
    \item[(i)] the associated tree is $\tilde{\alpha}$-regular with $\tilde{\alpha} = {b\rho}/{(8M)}$,
    \item[(ii)] the shape regularity factor ${h_+(\mathcal V _ N (x) )}/{h_-(\mathcal V_N (x) )}$ is bounded from below by $(1-\rho)^{-\sqrt{N/d}}.$ 
\end{itemize} 
\end{theorem}

Note that the factor $(1-\rho)^{-\sqrt{N/d}}$ diverges as $N \to \infty$. Observe that, for the choice $N = \log(n)/(2 \log(M/(b\rho)))$, the condition on $n$  and $N$ from the above theorem becomes $$ 16\log(192 \nu ^4 d (2n+1)^{2d} ) \leq \sqrt{n},$$ which is easily satisfied for $n$ large enough. Let us also mention that the order of $\log(n)$ for this choice of $N$ is standard as observed in \cite{biau2012analysis,wager2018estimation}.
%Furthermore, observe that condition \ref{cond:Wager} is satisfied with $\pi = 1$.

%The result established in Theorem \ref{Wager_subopt} provides a quantitative basis for the suboptimality of a specific class of "data-independent" trees. 

% It is crucial to emphasize that this theorem targets trees where the split positions are determined independently of the cutting directions. While the $\alpha$-regularity condition effectively prevents the formation of "empty" cells by ensuring that each split retains a minimum proportion of points, 

%we demonstrate that it is fundamentally insufficient to preserve the geometric integrity of the partition in this setting. 

The central contribution of this theorem is to prove that with high probability, there are some trees, that both satisfy \ref{cond:Wager} and are such that for any fixed point $x \in [0,1]^d$, the cell aspect ratio -- defined as the ratio between the largest and smallest sides of the cell containing $x$ -- grows at an exponential rate of order $\exp(\sqrt{N})$. Since this lower bound tends to infinity with $N$, this highlights a structural instability in this subclass of trees. This exponential divergence as $N$ increases suggests that the estimator cannot achieve the optimal rate (see Section \ref{leading_ex} and in particular \rev{Proposition} \ref{contre_ex}). Consequently, it shows that the random split condition is not sufficient to guarantee the shape regularity of the cells, and is likely insufficient to achieve the optimal rate.

This phenomenon of geometric instability stems from the absence of a correction mechanism: since axes are chosen uniformly at random without regard to the cell's current geometry, the algorithm cannot guarantee that elongated cells will be split along their longest side to restore their balance. Cutting directions are selected without considering either the current shape of the cell or the data distribution. Unlike adaptive algorithms such as CART or Mondrian trees -- which can ``correct'' a cell's elongation by splitting along its longest axis -- the random split mechanism is missing this feedback mechanism. %This inherent lack of feedback underscores why a dedicated constraint on cell geometry, namely shape regularity, is indispensable for achieving optimal statistical efficiency.

%Furthermore, geometric instability can be interpreted as a waste of the splitting budget $N$. By choosing axes randomly, the algorithm allocates splits to dimensions that are already compressed, allowing other dimensions to expand without control.

%This lack of coupling between cell geometry and observations drastically limits the adaptivity of the estimator. Indeed, the tree structure lacks any mechanism to correct the previously highlighted shape imbalances, which accumulate through successive splits. By verifying that the moment conditions of Corollary \ref{cor_moments} are satisfied under the specific constraints of this model, and by confirming that we indeed have a retention of a proportion $\alpha$ of points per cell ($\alpha$-regularity), we obtain the following result.

The analysis of this subclass of trees serves as a primary example for understanding the hierarchy of conditions necessary for optimal convergence. 
%In the decision tree literature, the balance of a tree is often reduced to its combinatorial aspect (the number of points per cell). 
Our results highlight a key distinction: random split condition, even when coupled with $\alpha$-regularity does not imply shape regularity. %Shape regularity requires that cells do not flatten excessively in any given direction, ensuring that the cell diameter $\text{diam}(\mathcal{V}_N(x))$ decreases isotropically with the depth $N$. In the example studied here, although the algorithm is $\alpha$-regular (each child cell receives a significant fraction of the points), the lack of coordination between the cutting direction $D_i$ and the geometry of the parent cell leads to a geometric drift. 
This directly supports the introduction of a new set of rules to build trees with statistical guarantees, as proposed in the next section.

\subsection{Shape regular tree}\label{s53}

 In this section, we introduce \textit{shape regular trees} (SR trees), as a tree construction that incorporates geometric constraints into the splitting process. It ensures that the resulting partition remains shape-regular, thereby inheriting the optimal convergence properties established previously.

%While the previous section established bounds that adapt to the function's directional smoothness, we now consider the case where geometric constraints are imposed on the cells. Specifically, we study the impact of shape regularity -- a condition ensuring that cells remain somewhat "square" -- on the global convergence rate. More precisely, in this section, 
 %analyze the performance of a general tree under two additional regularity constraints. 
 SR trees are constructed with two important conditions.
 First, a shape regularity condition is imposed to link the diameter of the cells to their volume, ensuring that they do not become too elongated. Second, we require a minimum number of points per leaf, which ensures that each cell remains statistically representative and effectively controls the estimator's variance. 
 We consider general regression trees for which each split is selected using a general cost function. In particular, the deviation inequality obtained below is valid for partitions that may depend on the whole dataset $(X_i,Y_i)_{i=1,\ldots,n }$ and not only on the covariates. %We call these trees ``CART-like'', since CART algorithm is arguably the most important instance of such data-dependent regression trees, due to its wide fame and use in practice. 

Let us introduce a general class of recursive data dependent trees. For simplicity, we assume that  $S_X = [0,1]^d$, as in the previous section. 
For a given cell $V$, a split is characterized by two parameters $(p, u) \in S : = \{ 1,\ldots, d\}\times  (0,1)$. Recall that for a cell $V$, we denote by $h_k(V)$ the (Euclidean) length of its side along coordinate $k$. The resulting left and right child cells, $V(l)$  and $V(r)$, are such that for any $k \neq p$, $h_k(V(l))=h_k(V(r))  = h _k (V) $, and for $k=p$, $ h_k(V(l)) = h_k(V) u$ and $ h_k(V(r)) = h_k(V) (1-u) $. We also recall that $h_-(V) = \min_{k = 1,\ldots, d} h_k(V) $ and $h_+(V) = \max_{k = 1,\ldots, d} h_k(V) $. With these notations, the split condition for $V$ to be $\beta$-shape regular can be expressed with the help of a restriction on the set of valid splits. Given $V$, let us define the set of $\beta$-shape regular splits as follows,
$$ S_\beta(V) : = \{  (p,u)  \in S  \  : \  h_+( V(s)) \leq \beta h_-( V(s)),\,  \forall s  \in \left\{ l,r \right\}   \}. $$
We note that when $\beta \geq 2$, $S_\beta(V)$ cannot be empty. Splitting the largest side in the middle is always in $ S_\beta(V) $. Another restriction on the splits is needed to ensure a sufficient number of points. It is given by
$$ S_m(V) : = \{  (p,u) \in S  \  : \  n P_n^X( V(s) ) \geq m , \, \forall s  \in \left\{ l,r \right\}   \} .$$
We do not need to fully specify the splitting criterion. When $S_m(V)\neq \emptyset$, the split in the cell $V$ is defined as a minimizer -- assumed to exist -- on $S_\beta(V)  \cap S_m(V)$, of a cost function $M_n$, given by
\begin{align*}
    M_n\, : \, S \times \mathcal R ([0,1]^d)  \to &\, \mathbb R \\
    ((p,u), V) \mapsto &\,  M_n ((p,u), V),
\end{align*}
where $ \mathcal R ([0,1]^d)$ is the set of hyper-rectangles included in $[0,1]^d$. In the case where $S_m(V) = \emptyset$, no split is performed and the cell $V$ remains unchanged. The main strength of our analysis lies in the generality of the cost function, which can actually be any function ensuring the existence of a minimizer as required above, and that may depend or not on the sample. For instance, in CART-regression, the cost function depends on the sample and is defined as
$$ M_n ((p,u), V) = \frac{\sum_{i=1}^n (Y_i - \overline Y( V(l))) ^2 \, \ind_{V(l)}(X_i)}{n P_n^X (V(l) )  }  + \frac{\sum_{i=1}^n (Y_i - \overline Y(V(r)))^2 \, \ind_{V(r)}(X_i)}{n P_n^X(V(r) )   } $$
where $ \overline Y(V )= \sum_{i=1}^n Y_i \, \ind_{V}(X_i) / (n P_n^X (V ))$ for any cell $V$. 

\begin{algorithm}[htbp]
\begin{algorithmic}[h]
\Statex{\textbf{Input:} Sample  $(X_i, Y_i)_{i=1,\ldots, n} \subset [0,1]^d \times \mathbb R$, minimal number of points $m\in \{1,\ldots, n\} $, shape-regularity $\beta \geq 2$, cost function $M_n: S \times \mathcal R ([0,1]^d)  \to   \mathbb R$.
Let $V^{(0)} = \{[0,1]^d\} $ be the initial partition, made of one element (i.e. $| V^{(0)}|=1$).}
   \For{$j= 0,1,\ldots$}
\Statex\hspace{\algorithmicindent}{Let $V^{(j+1)} = \emptyset$} denote the partition at step $j+1$. The update is as follows: 
   \For{$k=  1,2,\ldots, | V^{(j)}|$} 
\Statex\hspace{\algorithmicindent}\hspace{\algorithmicindent}
(a) Whenever $S_m(V^{(j)}_k) \neq \emptyset $, define two children, $V(l)$ and $V(r)$, according to
$$\argmin_{(p,u) \in S_\beta(V^{(j)}_k)\cap S_m(V^{(j)}_k) } M_n( (p,u) , V^{(j)}_k) $$
\hspace{\algorithmicindent}\hspace{\algorithmicindent}{ If the above optimization problem has no solution, just pick $p$ as the largest side} 
\Statex\hspace{\algorithmicindent}\hspace{\algorithmicindent}{ and $u = 1/2$. Set $$V^{(j+1)} = \{ V^{(j+1)}, V (l), V (r)\}$$
}
\Statex \hspace{\algorithmicindent}\hspace{\algorithmicindent}{
(b) Whenever $S_m(V^{(j)}_k) = \emptyset $, child is same as parent. Set
  $$V^{(j+1)} = \{ V^{(j+1)}, V^{(j)}_k \}$$
}
\EndFor
\Statex\hspace{\algorithmicindent}{STOP if $ V^{(j+1)} = V^{(j)} $ (no valid split exists)}
   \EndFor
\Statex{Return the final partition elements $V^{(j+1)}$}
\end{algorithmic}
\caption{SR trees}
\label{alg:cart-like}
\end{algorithm}

By splitting on the intersection of $S_\beta$ and $S_m$,  Algorithm \ref{alg:cart-like} ensures that the two conditions are met when growing the tree. The first growing condition, which is the $\beta$-shape regularity of the cell, may not constitute a stopping criterion. Indeed, because $\beta\geq 2$, one can always split at the middle the largest side of the considered cell. The other growing condition on $m$ is easy to check in practice since it amounts to keep a cell as a leaf if and only if the number of data points belonging to that cell is greater than $m$ and strictly smaller than $2m$. As a consequence, one might modify classical algorithms, in the case precisely where the split proposed by the algorithm does not respect the $\beta$-shape-regularity condition for a prescribed value of $\beta$, or the other growing condition asking for sufficiently many points in the cells.

The next theorem gives a deviation inequality on the error associated to the regression map estimator resulting from Algorithm \ref{alg:cart-like}.

\begin{theorem}\label{th_cart_like1}
Let  $\delta \in (0,1/3) $, $n\geq 1$, $d\geq 1$, $\beta\geq 2$ and $m\in \{1,\ldots,  n\}$ such  that $m\geq  4\log(4(2n+1)^{2d}/\delta )$. Suppose that \ref{cond:D}, \ref{cond:density_XCART}, \ref{cond:epsilon} and \ref{cond:reg4} are fulfilled. Let $\mathcal V$ be the local regression map obtained from a CART-like tree with input parameters $\beta $, $m$ and cost function $M_n$, then we have, with probability at least $1-3\delta$, for all $x \in [0,1]^d$,
    \begin{align*} 
            |  \hat g_\mathcal V (x) - g(x)  | \leq \sqrt { \frac{2 \sigma ^2  } {  m}\log\left( \frac{(n+1) ^{2d}}{ \delta } \right)} + L(\mathcal{V}(x))  \beta \sqrt{d} \left(\frac{5 m}{ n b}\right)^{1/d}.
    \end{align*}
\end{theorem}

Note that the conditions on the value of $m$ are satisfied whenever $n$ is sufficiently large and $m\asymp n^a $, 
 for any $a\in (0,1)$. Notice that taking $m \asymp n^{2/(d+2)}$ in the estimation bound of Theorem \ref{th_cart_like1} gives the optimal convergence rate $n^{-1/(d+2)}$, up a multiplicative logarithmic term. Moreover, such a value of $m$ allows the bound to be valid with a probability that grows to one polynomially in $n$, since the constraint $m \geq 4\log(4(2n+1)^{2d}/\delta )$ will be then satisfied.
% and the term $\log(4(2n+1)^{2d}/\delta )$ must stay logarithmic in $n$. 
In addition, such results remain valid for the rate of convergence in sup-norm whenever the density $f_X$ is uniformly bounded from below by a positive constant, independent of $n$. This is stated in the subsequent corollary.

\begin{corollary}\label{cor_cart_like1}
In Theorem \ref{th_cart_like1}, if the integer $m$ is chosen as $m \asymp n^{2/(d+2)} \log((n+1)^{2d}/\delta)^{d/(d+2)}$, then we have the following inequality for $n$ sufficiently large with probability at least $1-3\delta$, 
   \begin{align*}
      \sup_{x\in [0,1]^d} | \hat g_{\mathcal V}(x) - g(x)|  \lesssim c \left( \frac{1}{n} \log\left(\frac{(n+1)^{2d}}{\delta}\right)  \right)^{1/(d+2)},
   \end{align*}
  where $c =   \sqrt{2\sigma^2} + \beta L \sqrt{d} ({5}/ b )^{1/d}$.
\end{corollary}

The previous result shows that SR trees are able to attain the optimal rate of convergence as soon as a simple constraint -- restricting acceptable splits by a simple rule -- is imposed during the tree construction. 

Interestingly, results presented in \cite{cattaneo2022pointwise} tend to indicate that such modifications are in general necessary for the classical CART algorithm to achieve a good pointwise -- or uniform -- behavior. More precisely, it is shown in \cite{cattaneo2022pointwise} that the use of CART is problematic for the estimation of a constant regression function,  measured with the sup-norm error. Indeed, its rate of convergence in dimension one is slower than any polynomial of the sample size $n$, with non-vanishing probability. In addition, the honest version of CART -- i.e. when the prediction values among the cells use data that are independent of those used to construct the partition (see Definition 5.1 in \cite{cattaneo2022pointwise}) -- is proved to be inconsistent with positive probability as soon as the tree depth is of order at least $\log(\log(n))$.  This is due to the fact that the splitting criterion produces leaves that are too small.  

Our results complete the picture drawn in \cite{cattaneo2022pointwise} by putting forward the fact that producing too small cells is \textit{the only problem} that can occur with the use of CART in dimension one. Indeed, any cell being $\beta$-shape-regular in dimension one, with $\beta=1$, Theorem \ref{th_cart_like1} shows that the only problem must come from the \rev{amount} of data $m$ in the least populated cell. Indeed, if $m$ is of order $\log(n)$, then our deviation bound in Theorem \ref{th_cart_like1} does not converge to zero when $\delta$ is fixed and the sample size goes to infinity. This is basically what happens in \cite{cattaneo2022pointwise}. In such a case, we are indeed not able to prove the consistency of CART.

\section{Perspective on the anisotropic analysis of trees}\label{sec:perspective}

The efficiency of recursive partitioning algorithms, such as CART, relies on their ability to adapt the geometry of the cells to the local variations of the regression function. To mathematically capture this property, it is necessary to move away from analyses based on a global diameter in favor of a directional approach. We therefore introduce the coordinate-wise Lipschitz regularity assumption, which allows us to distinguish the influence of each variable on the variations of $g$.
\begin{enumerate}[label=(CL), wide=0.5em, leftmargin=*]
\item \label{cond:lip_direc} The regression function $g$ is coordinate-wise Lipschitz, meaning there exist constants $L_1, \dots, L_d \geq 0$ such that, for all $(u, v)$ in $ [0,1]^d$, $$|g(u) - g(v)| \leq \sum_{j=1}^d L_j |u_j - v_j|.$$
\end{enumerate}

For all $j \in \{1, \dots, d\}$, let $h_j(\mathcal{V}(x))$ denote the side length of the cell $\mathcal{V}(x)$ along the $j$-th dimension. We also introduce $L_j(\mathcal{V}(x))$, the local Lipschitz constant of the target function $g$ within the cell $\mathcal{V}(x)$ along the $j$-th direction. Building upon this definition, we claim a refined risk bound that explicitly accounts for the local sample size and the adaptive cell dimensions, providing a sharp characterization of the estimator's performance.
\begin{theorem}\label{CART_pratique}
Let $n\geq 1$ and $\delta \in (0,1/2) $. Under assumptions \ref{cond:epsilon}, \ref{cond:D}, and \ref{cond:lip_direc}, the following inequality holds with probability at least $1 - 2\delta$, for all $x \in [0,1]^d$,
\begin{align*}
       |  \hat g_{\mathcal V}(x) - g(x)| \leq \sqrt{\frac{ 2 \sigma^2 }{ n P_n^X(\mathcal V (x))  }\log\left( \frac{ (n+1) ^{2d} }{\delta} \right) } + \sum_{j=1}^d L_j(\mathcal{V}(x))  h_j(\mathcal{V}(x)).
    \end{align*}
\end{theorem}

The previous theorem might be of  practical interest since the deviation bound involves quantities
%aspect of this theorem lies 
%in the use of the empirical measure $P_n^X(\mathcal{V}(x))$, which corresponds exactly to the proportion of training points falling into a given leaf.  
directly computable from the considered leaf: the leaf's sample size and its geometric dimensions. Applying this bound in practice would further require estimating several quantities that are typically unknown. First, the noise variance $\sigma^2$ must be estimated. This has been the object of recent studies \cite{devroye_variance,ramosaj2019consistent}. % -- which can be localized and replaced by $\sigma^2(\mathcal{V}(x))$ to account for potential heteroscedasticity -- 
Second, the directional Lipschitz constants $L_j(\mathcal{V}(x))$ might  be evaluated following local linear methods \cite{fan1996}. %We note that encouraging sp The estimation of  $L_j(\mathcal{V}(x))$ might also be done encouraging sparsity as in \cite{ausset2021nearest}. 
%This approach theoretically justifies CART's greedy strategy, which, by seeking to minimize local impurity, prioritizes reducing the diameters $h_j$ associated with the steep slopes of $g$. By reducing $h_j$ where $L_j$ is large, the algorithm precisely minimizes the bias term, while preserving the variance term by avoiding unnecessary splits. 

    From a theoretical stand point, Theorem \ref{CART_pratique} directly links the geometric structure of the tree's leaves to the tree's statistical performance. %By substituting the global diameter with the weighted sum of side lengths $\sum_{j=1}^d L_j(\mathcal{V}(x)) h_j(\mathcal{V}(x))$, it highlights CART's natural capacity to adapt to the structure of the target function $g$. 
Since the bias error only depends on the directions where the function varies significantly, a good algorithm would afford not to divide the space along dimensions where the Lipschitz constants $L_j$ are null or negligible. This means that the shape of tree cells does not need to be isotropic or ``cubic'' to be efficient. On the contrary, the theorem encourages the use of anisotropic partitions that save splits in non-informative directions, thus preserving a higher count $n P_n^X(\mathcal{V}(x))$ to reduce variance. 

\rev{This suggests the use of splitting rules different from the SR condition in order to build the tree.}
Those rules would reveal cells with small gradient to escape the curse of dimensionality by focusing on the intrinsic dimension of the regression problem. Such directions of research would go beyond the scope of the present paper and are left as interesting questions for future work.

 %Finally, this result provides a theoretical foundation for pruning strategies, as it introduces a stopping rule that carefully weighs the trade-off between the geometric refinement of the partition and the statistical reliability of each leaf.

\newpage

\bibliography{b2}

\newpage
\appendix

\section{Proof of the results stated in Section \ref{s30} and \ref{sectionSR}}

Let $\PP$ be the probability measure on the underlying probability space $(\Omega, \mathcal F)$ on which are defined all introduced random variables.

\subsection*{Proof of Theorem \ref{1}}

Let \( \mathbb P_{X_{1:n}} \) denote the conditional probability given \( X_1, \dots, X_n \). Let $\mathcal V = \{ \mathcal V (x) \, : \, x\in \mathbb R^d\}$ and define
$$ \mathcal G = \{ (\mathds 1 _{ A }(X_1), \dots,\mathds 1 _{ A}(X_n)) \, :\,  A \in \mathcal{A}  \}.$$
With this notation we have 
$$\sup_{x\in \mathbb R^d} \frac{\sum_{i=1}^n \varepsilon_i \mathds 1 _{ \mathcal V(x) }(X_i)}{\sqrt{\sum_{j=1}^n \mathds 1 _{ \mathcal V(x) }(X_j)}} \leq  \sup_{(g_1,\ldots, g_n) \in \mathcal G } \frac{\sum_{i=1}^n \varepsilon_i g_i }{\sqrt{ \sum_{j=1}^n g_j}}.$$
Consequently, for all $t > 0$,
\begin{align*}
   \pr_{X_{1:n}}\left(\sup_{x\in \mathbb R^d} \dfrac{\sum_{i=1}^n \varepsilon_i \mathds 1 _{ \mathcal V(x) }(X_i)}{\sqrt{\sum_{j=1}^n \mathds 1 _{ \mathcal V(x) }(X_j)}} > t  \right) 
   &\leq \pr_{X_{1:n}}\left( \bigcup_{(g_1,\ldots, g_n) \in \mathcal G }  \left\{  \frac{\sum_{i=1}^n \varepsilon_i g_i }{\sqrt{ \sum_{j=1}^n g_j}} > t \right\}  \right) \\ 
   &\leq \sum_{(g_1,\ldots, g_n) \in \mathcal G } \pr_{X_{1:n}}\left(\frac{\sum_{i=1}^n \varepsilon_i g_i}{\sqrt{ \sum_{j=1}^n g_j}} > t  \right).
\end{align*} 
Moreover, since the conditional distribution of $\varepsilon_i$ given $X_1,\ldots,X_n$ is sub-Gaussian with parameter $\sigma^2$, then  $\varepsilon_i g_i$ is sub-Gaussian under $\pr_{X_{1:n}}$, with parameter $\sigma^2 g_i^2$. Hence, by conditional independence given $(X_i)_{i = 1,\ldots, n}$, ${\sum_{i=1}^n \varepsilon_i g_i}/{\sqrt{ \sum_{j=1}^n g_j}}$ is sub-Gaussian with parameter $\sigma^2 \sum_{i=1}^n g_i^2 / \sum_{j=1}^n g_j$. Indeed, $(Y_i) _{i = 1,\ldots, n} $ (and so $(\varepsilon_i) _{i = 1,\ldots, n} $) is an independent collection of random variables, conditionally on $(X_i)_{i = 1,\ldots, n}$. We prove this fact in lemma \ref{lemmaA} below.  Moreover, $\sum_{i=1}^n g_i^2 = \sum_{i=1}^n g_i$ because $g_i\in \{ 0,1\} $. Hence, ${\sum_{i=1}^n \varepsilon_i g_i}/{\sqrt{ \sum_{j=1}^n g_j}}$ is sub-Gaussian with parameter $\sigma^2$ under $\pr_{X_{1:n}}$. Therefore, we obtain     
\begin{align*}
    \pr_{X_{1:n}}\left(\sup_{x\in \mathbb R^d} \dfrac{\sum_{i=1}^n \varepsilon_i \mathds 1 _{ \mathcal V(x) }(X_i)}{\sqrt{\sum_{j=1}^n \mathds 1 _{ \mathcal V(x) }(X_j)}} > t  \right)  &\leq \sum_{(g_1,\ldots, g_n) \in \mathcal G } \exp\left( \dfrac{-t^2}{2\sigma^2}\right)  \leq  \mathbb S_\mathcal V(n) \exp\left( \dfrac{-t^2}{2\sigma^2}\right).
\end{align*}
   If we set $\delta = \mathbb S_\mathcal A(n) \exp\left( {-t^2} /{(2\sigma^2)}\right)$, we have $t = \sqrt{2\sigma^2 \log\left({\mathbb S_\mathcal A(n)}/{\delta}\right)}$. Finally, with probability $\pr_{X_{1:n}}$ at least equal to $1 - \delta$, we get
$$\sup_{x\in \mathbb R^d} \dfrac{\sum_{i=1}^n \varepsilon_i \mathds 1 _{ \mathcal V(x) }(X_i)}{\sqrt{\sum_{j=1}^n \mathds 1 _{ \mathcal V(x) }(X_j)}}  \leq \sqrt{2 \sigma^2 \log\left( \frac{\mathbb S_{\mathcal A}(n)}{\delta} \right)}.$$
   Since $\delta$ is independent of $(X_1,\ldots,X_n)$, we obtain the result by integrating with respect to $(X_1,\ldots,X_n)$. \qed 

\subsection*{Proof of Theorem \ref{th:general}}

    Let \(x \in S_X\). We write the bias-variance decomposition 
$\hat g_{\mathcal V}(x) - g(x) = V + B$, where \begin{align*}
    V := \dfrac{\sum_{i=1}^n \varepsilon_i \mathds 1 _{ \mathcal V(x) }(X_i)}{\sum_{j=1}^n \mathds 1 _{ \mathcal V(x) }(X_j)} \qquad \text{and}\qquad
       B := \dfrac{\sum_{i=1}^n \left(g(X_i) - g(x)\right) \mathds 1 _{ \mathcal V(x) }(X_i)}{\sum_{j=1}^n \mathds 1 _{ \mathcal V(x) }(X_j)}.
\end{align*}
The inequality from Theorem \ref{1} gives, with probability at least $1 - 2\delta$, for all $x \in S_X $,
\begin{align*}
    |V| &\leq \left({\sqrt{\sum_{j=1}^n \mathds 1 _{ \mathcal V(x) }(X_j)}} \right)^{-1} \sup_{x\in \mathbb R^d} \left| \dfrac{\sum_{i=1}^n \varepsilon_i \mathds 1 _{ \mathcal V(x) }(X_i)}{\sqrt{\sum_{j=1}^n \mathds 1 _{ \mathcal V(x) }(X_j)}} \right| \\ &\leq \dfrac{1}{\sqrt{n P_n^X(\mathcal{V}(x))}} \ \sqrt{2 \sigma^2 \log\left( \frac{\mathbb S_\mathcal A(n)}{\delta} \right)} .
\end{align*}
Using the inequality $\mathbb S_\mathcal A(n) \leq (n+1)^v$ we recover the first term of the stated bound. Furthermore, using the triangle inequality, we obtain that
\begin{eqnarray*}
    |B| &\leq& \dfrac{\sum_{i=1}^n \left|g(X_i) - g(x)\right| \mathds 1 _{ \mathcal V(x) }(X_i)}{\sum_{j=1}^n \mathds 1 _{ \mathcal V(x) }(X_j)} \\ &\leq& \dfrac{\sum_{i=1}^n \sup_{y \in \mathcal V(x)} |g(y) - g(x)| \mathds 1 _{ \mathcal V(x) }(X_i)}{\sum_{j=1}^n \mathds 1 _{ \mathcal V(x) }(X_j)} = \sup_{y \in \mathcal V(x)} |g(y) - g(x)|.
\end{eqnarray*}
Moreover, using the Lipschitz assumption, it follows that
$$|g(y) - g(x)| \leq  L(\mathcal V(x)) \|x-y\|_2  \leq L(\mathcal V(x)) \diam  (\mathcal V(x) ) , $$
which concludes the proof. \qed
%\subsection{Proofs related to Section \ref{3.1}}

\subsection*{Proof of Proposition \ref{contre_ex}}

Let $V_0 =  \mathcal V(0) $. Define
$$ W =  \frac{\sum_{i=1} ^n  (Y_i - g(X_i) ) \ind_{ V_0} (X_i)}{\sum_{i=1} ^n    \ind_{V_0 } (X_i)} =  \frac{\sum_{i=1} ^n  \varepsilon_i \ind_{ V_0} (X_i)}{\sum_{i=1} ^n    \ind_{V_0 } (X_i)}$$ and 
 $$B = \frac{\sum_{i=1} ^n  g(X_i)   \ind_{V_0 } (X_i)}{\sum_{i=1} ^n    \ind_{V_0} (X_i)}. $$
Denote by $E$ the event $\left \{ \sum_{i=1} ^n    \ind_{V_0} (X_i) > 0 \right \}$, and the event $\{ \sum_{i=1} ^n    \ind_{V_0} (X_i) = 0  \} $ by $\bar{E}$. We have, since $g(0) = 0$, 
$$\hat g_{\mathcal V}(0) - g(0) = 0 \times \ind_{\bar{E}} + (W + B) \ind_{E} , $$
and since $E$ is $X$-measurable and $\mathbb{E}(\varepsilon | X)=0$, we have
$$ \EE [ (\hat g_{\mathcal V} (0 ) -  g(0 ))^2 ]  = \EE[W^2  \ind_{E} ] + \EE [B^2  \ind_{E} ]. $$ 
According to Lemma \ref{lemmaA}, the variables $(\varepsilon_i)_{i=1,\dots,n}$ are independent conditional on $(X_i)_{i=1,\dots,n}$. Furthermore, given the equality $\EE(\varepsilon|X) = 0$ and the inequality $\EE(\varepsilon^2 | X) \geq \sigma^2_{\text{min}}$, we obtain
$$ \EE [W^2 \ind_{E} |X_1,\ldots, X_n] = \frac{\sum_{i=1} ^n  \EE(\varepsilon^2 | X)  \ind_{ V_0} (X_i)}{\left(\sum_{i=1} ^ n \ind_{ V_0 } (X_i)\right)^2 } \ind_{E} \geq \frac{\sigma^2_{\text{min}}}{\sum_{i=1} ^ n \ind_{ V_0 } (X_i) } \ind_{E}.$$
To obtain a deterministic bound, we integrate the previous lower bound with respect to the distribution of the sample. By definition of the conditional expectation, we have $$\EE \left[ \frac{1}{\sum_{i=1}^n \ind_{V_0}(X_i)} \ind_E\right] = \EE \left[ \frac{1}{\sum_{i=1}^n \ind_{V_0}(X_i)} \;\middle|\; E \right] \PP(E).$$ 
Considering the convexity of the function $\phi : x \mapsto 1/x$ on $\mathbb{R}_+^*$, Jensen's inequality applied to the conditional probability measure $\PP(\cdot | E)$ ensures that
$$\EE \left[ \frac{1}{\sum_{i=1}^n \ind_{V_0}(X_i)} \;\middle|\; E \right] \geq \frac{1}{\EE\left[\sum_{i=1}^n \ind_{V_0}(X_i) \mid E \, \right]}.$$
Under the assumption that the independent variables $X_i$ are uniformly distributed on $[0,1]^d$, the random variable $\sum_{i=1}^n \ind_{V_0}(X_i)$ follows a binomial distribution $\mathcal{B}(n, \lambda(V_0))$. The conditional expectation on the event $E$ is given by $$\EE\left[\sum_{i=1}^n \ind_{V_0}(X_i) \mid E\right] = \frac{\EE\left[\sum_{i=1}^n \ind_{V_0}(X_i) \ind_E\right]}{\PP(E)} = \frac{n \lambda(V_0)}{\PP(E)}$$ since the sum is identically zero on the complementary event $\bar{E}$. Substituting this result into Jensen's inequality, we obtain $$\EE \left[ \frac{1}{\sum_{i=1}^n \ind_{V_0}(X_i)} \ind_E\right] \geq \frac{\PP(E)^2}{n \lambda(V_0)},$$
and finally $$\EE[W^2  \ind_{E} ] \geq \frac{\sigma^2_{\text{min}} \, \PP(E)^2}{n \lambda(V_0)}.$$

Let  $V_1 = \prod_{k=1}^d [h_k/2 , h_k] \subset V_0$. We have, $a_0: = \sum_{i=1} ^n \ind_{V_0} (X_i)\geq \sum_{i=1} ^n \ind_{{ V_1}} (X_i) := a_1$. It follows that
\begin{align*}
   B \ind_{E} \geq \frac{\sum_{i=1} ^n g(X_i) \ind_{V_1} (X_i)}{\sum_{i=1} ^n \ind_{V_0} (X_i)} \ind_{E} 
    &\geq \frac{1}{2} ( h_1+ \dots + h_d ) \frac{a_1}{a_0} \ind_{E} \\ &\geq \frac{c}{2}  \diam_1( V_0) \ind_{a_1\geq c a_0 > 0}
\end{align*}
where the previous inequality is valid for any $c>0$. This implies that, for any $c>0$,
$$ \EE[B^2 \ind_{E}  ] \geq \frac{c^2}{4}  \diam _1 ( V_0) ^2 \,  \PP(  {a_1\geq ca_0 > 0 } ) \geq \frac{c^2}{4}  \diam _2 ( V_0) ^2 \,  \PP(  {a_1\geq ca_0 > 0 } ) .$$
Let us now look for a suitable choice of constant \( c > 0\). From Theorem \ref{lemma=chernoff}, one has that with probability at least $1 - 2\delta = 1/2$,
$$ a_1 \geq \frac{P^X({ V_1})}{P^X( V_0)} \frac{\left( 1 - \sqrt {2\log(4 ) / (n P^X( V_1)) }  \right )}{\left( 1 + \sqrt {3\log(4 ) / (n  P^X({  V_0 })})  \right )} \,  a_0.$$ Furthermore, note that $\lambda({ V_1}) = \prod_{k=1}^d h_k/2 = 2^{-d} \prod_{k=1}^d h_k = 2^{-d} \lambda( V_0)$ and $P^X({  V_k}) = \lambda({  V_k})$ for each  $k \in \{0,1\}$. Note also that we necessarily have $n \prod_{k=1}^d h_k \geq  2^{d+3} \log(4) \geq 3 \times 2^{d+1} \log(4) \geq 3 \times 4 \log(4).$ This ensures also that the numerator is positive. As a consequence, we find that, with probability at least $1/2$,
\begin{align*}
    a_1 &\geq 2^{-d} \, \dfrac{1 - \sqrt{\dfrac{2^{d+1} \log(4)}{n \prod_{k=1}^d h_k}}}{1+ \sqrt{\dfrac{3 \log(4)}{n \prod_{k=1}^d h_k}}} \, a_0 \geq 2^{-d} \dfrac{1 - 1/2}{1 + 1/2} \, a_0 = \dfrac{2^{-d}}{3} a_0 = c \, a_0.
\end{align*}
Moreover, we have $$\PP(a_0 = 0) =  (1-\lambda(V_0))^n \leq  \exp(-n\lambda(V_0)),$$ and by the hypothesis on $n$ we optain $$\PP(a_0 = 0)  \leq \exp(-2^{d+3}\log(4)) =  \left( \dfrac{1}{2} \right)^{2^{d+4}} \leq \dfrac{1}{4}.$$ Then $$\PP(a_1 \geq c a_0 > 0) = \PP(a_1 \geq c a_0) - \PP(a_0 = 0) \geq \dfrac{1}{2} - \dfrac{1}{4} = \dfrac{1}{4}.$$ Furthermore, we have also $\PP(E) = 1- \PP(a_0 = 0) \geq 1- 1/4 = 3/4.$ Thus, we have obtained that
\begin{align*}
    \EE [ (\hat g_{\mathcal V} (0 ) -  g(0 ))^2 ]  &= \EE[W^2 \ind_{E} ] + \EE [B^2 \ind_{E} ] \\ 
    &\geq \frac{9}{16}\dfrac{\sigma^2_{\text{min}}}{n \lambda (  V_0)} + \dfrac{ c ^2}{4} 
     \diam_2 ( V_0)^2 \times \dfrac 1 4  \\
    &\geq \dfrac{1}{16} \left(\dfrac{9\sigma^2_{\text{min}}}{n \lambda (  V_0)} + {(c \gamma)^2} \lambda ( V_ 0)^{2/d} \right) 
\end{align*}
where $\gamma = \overline \gamma ^{1/d} $. Let \( a_1 \) and \( a_2 \) be positive real numbers. By studying the function \( \psi : x \mapsto a_1x^{-d} + a_2 x^2 \) on \(\mathbb R_+^*  \), we notice that \( \psi \) has global minimum achieved at \( x_m = (a_1d/(2a_2))^{1/(d+2)} \). This implies that \begin{align*}
\min_{x > 0} \psi (x) &\geq x_m^2 a_2 \left( \frac{ a_1 } {a_2x_m^{d +2 } }  + 1\right) 
\\ &= \left(\dfrac{a_1d}{2a_2}\right)^{2/(d+2)} a_2 \left( \dfrac{2}{d} +1  \right) 
\\ &= \left(\dfrac{a_2^{d/2}  a_1  d}{2}\right)^{2/(d+2)}  \left( \dfrac{2}{d} +1 \right)
\end{align*}
Now, setting \( a_1 = 9 \sigma^2_{\text{min}} n^{-1} \), \( a_2 = (c \gamma)^2 \), we find \begin{align*}
    \EE [ (\hat g_{\mathcal V} (0 ) -  g(0 ))^2 ] &\geq \dfrac{1}{16}\psi ( \lambda (\mathcal V)^{1/d}) \\ &\geq \dfrac{1}{16}
    \left(\dfrac{9\sigma^2_{\text{min}} d \, (c \gamma)^{d}}{2n}\right)^{2/(d+2)} \left(1 + \dfrac{2}{d}\right) \\  &= C_d ^2 \left(\dfrac{\bar\gamma \sigma^2_{\text{min}}}{n}\right)^{2/(d+2)}
\end{align*} where $C_d =  \sqrt{{2}/{d} + 1} \, \left({ 9d}/{2}\right)^{1/(d+2)} \left( 3 \times 2^d\right)^{-d/(d+2)}/4 \,.$ \qed

\subsection*{Proof of Theorem \ref{th2:general}}
Assume that the maximum of $P_n^X(\mathcal V(x))$ and $P^X(\mathcal V(x))$ is $P^X(\mathcal V(x)).$
We have, by assumption, for all $x\in S_X$,
$$ { n     P^X(\mathcal V(x)) }  \geq 36 \log\left(\dfrac{4 (2n+1) ^v }{\delta} \right). $$
We deduce that
$$ \frac 2 3 \leq 1 -   \sqrt{ \frac{4 \log\left(\dfrac{4 (2n+1) ^v }{\delta} \right) }{n     P^X(\mathcal V(x))} }.$$
Hence, using Theorem \ref{th:vapnik_normalized}, we obtain that with probability $1-\delta$, for all $x\in S_X$,
$$ P_n^X(\mathcal V(x)) \geq P^X(\mathcal V(x)) \left(1-\sqrt { \frac{4 \log( 4 (2n+1) ^v / \delta) }{ n P^X (\mathcal V(x))} } \right)\geq \frac 2 3P^X (\mathcal V(x))  \geq \frac 2 3\ell(x) \lambda (\mathcal V(x) ).$$
Now, if the maximum of $P_n^X(\mathcal V(x))$ and $P^X(\mathcal V(x))$ is $P_n^X(\mathcal V(x))$, then we have
$$ P_n^X(\mathcal V(x) ) \geq P^X(\mathcal V(x) )  \geq \ell(x) \lambda (\mathcal V(x) ).$$ %\geq 2\ell(x) \lambda (\mathcal V(x) )/5
Using Theorem \ref{th:general} and the previous inequality on $ P_n^X(\mathcal V(x) )$ yields the result. \qed

\subsection*{Proof of Proposition \ref{link_beta_gamma}}

Let $A$ be a hyper-rectangle. We use the shortcut $h_- $ and $h_+$ for $h_- (A) $ and $ h_+(A)$, respectively. The first statement is a consequence of 
$ \diam (A) \leq \sqrt{d} h_+ $ 
and 
$ \lambda (A) \geq h_- ^d $, as using $\beta $-shape regularity, we obtain
$$ \diam (A)\leq \sqrt{d} \beta h_- \leq \sqrt{d} \beta \lambda (A) ^{1/d}.$$
    The second statement can be obtained as follows. Since $\diam (A) \geq h_ +$ and $ \lambda (A)^{1/d} \leq h_+ ^{1 -1/d}  h_- ^{1/d}$ we find
 \begin{align*}
      \gamma^{1/d} \geq \frac{ \diam (A)}{ \lambda (A)^{1/d} } \geq \frac{ h_ + }{  h_+ ^{1 -1/d}  h_- ^{1/d} } = \left(\frac{  h_ +}{   h_-  } \right)^{1/d}.
\end{align*}\qed

\subsection*{Proof of Theorem \ref{main_result_localizing_map}}
By assumption, there is $(a_-,a_+) $ such that $ 0< a _- \leq 1\leq   a _+ < +\infty $ and for all $x \in S_X$,
$$ \lambda(\mathcal V(x))  a_-  \leq \left( \frac{\log \left({(n+1)^v}/{\delta} \right)}{n} \right)^{d/(d+2)}\leq a_ + \lambda(\mathcal V(x)).$$
According to Theorem \ref{th2:general}, the $\gamma$-SR assumption, we obtain with probability at least $1 - 3\delta $, for all $x\in S_X$
\begin{align*}
        | \hat g_{\mathcal V}(x) - g(x)| &\leq  \sqrt{\frac{ 3 \sigma^2 \log\left( \frac{(n+1)^v}{\delta} \right) }{n  \ell(x) \lambda ( \mathcal V (x) )   }} + L(\mathcal V(x) ) \diam (\mathcal V(x) ) \\ 
        &\leq  \sqrt{\frac{ 3 \sigma^2 \log\left( \frac{(n+1)^v}{\delta} \right) }{n  \ell(x) \lambda ( \mathcal V (x) )   }} + L(\mathcal V(x) )\gamma^{1/d}  \lambda ( \mathcal V (x) )^{1/d} \\ &\leq  \sqrt{\frac{ 3 \sigma^2   \lambda ( \mathcal V (x) )  ^{(d+2)/d}a_+ ^{(d+2)/d}}{ \ell(x) \lambda ( \mathcal V (x) )   }} + L(\mathcal V(x) )\gamma^{1/d}  \lambda ( \mathcal V (x) )^{1/d} \\ &\leq  \left(  \sqrt{\frac{ 3 \sigma^2 a_+ ^{(d+2)/d} }{ \ell(x)}} + L(\mathcal V(x) )\gamma^{1/d} \right)  \lambda ( \mathcal V (x) )^{1/d} \\ &\leq \left( \sqrt{\frac{ 3 \sigma^2 a_+ ^{(d+2)/d}}{ \ell(x)}} + L(\mathcal V(x) )\gamma^{1/d} \right) 
 \left(\frac{\log\left( \frac{(n+1)^v}{ \delta} \right) }{ n} \right) ^{1/(d+2)} a_- ^{-1/d}.
\end{align*}
The result follows by taking care that $a_+ ^{(d+2)/d} \leq a_+^3 $  and $ a_- ^{-1/d}\leq a_-^{-1} $ which means that the universal constant in the upper bound can be taken as $ a_+^{3/2} / a_- $. \qed

\section{Proof of the results stated in Section \ref{sec_k_nn}.}

\subsection*{Proof of Theorem \ref{th:NN}}

For any $x\in S_X$, define $ \tau(x)^d  = 2 k /  (n \ell(x) ) $  and check that  $\tau(x)^d \leq T_0^d$. Using \ref{cond:density_XNN} we obtain 
   \begin{align*}
    \forall x\in S_X,\qquad   n P^X(B (x, \tau(x)) ) \geq n \ell(x) \tau(x)^d   = 2  k .
   \end{align*}
   Next from Theorem \ref{th:vapnik_normalized}, and using that the set of all balls in $\mathbb R^d$, denoted by $\mathcal{A}$, has Vapnik dimension $d+1$ so that $\mathbb S_\mathcal A(2n)\leq (2n+1)^{(d+1)} $,  we deduce that with probability at least $1-\delta$, for all $x \in S_X,$
   $$n P_n^X (B (x, \tau(x)) ) \geq n P^X(B (x, \tau(x)) ) - \sqrt { n P^X(B (x, \tau(x)) )  4\log( 4 (2n+1)^{(d+1)}  / \delta)   } . $$
   Note that $x\mapsto x - \sqrt{ x\ell} $ is increasing whenever $ x\geq \ell /4$. Since, by assumption on $k$, 
   $$\forall x\in S_X,\quad  n P^X(B (x, \tau(x)) ) \geq  2k \geq 16\log( 4 (2n+1)^{(d+1)}  / \delta) \geq \log( 4 (2n+1)^{(d+1)}  / \delta).   $$
   We obtain that, with probability at least $1-\delta$,
   $$ \forall x\in S_X,\qquad    n P_n^X (B (x, \tau(x)) ) \geq 2k - \sqrt {8k  \log( 4  (2n+1)^{(d+1)}  / \delta)   }  .$$
   Now using again that $k \geq  {  8 \log( 4 (2n+1)^{(d+1)} / \delta)   }  $, we find that with probability at least $1-\delta$
   $$ \forall x\in S_X,\qquad  nP_n^X(B (x, \tau(x)))\geq k .$$
   However, for each $x\in S_X$, $\hat \tau _{n,k}(x) $ is defined as the smallest such value of $\tau $. Therefore, we obtain that for all $x \in S_X$, $ \hat \tau_{n,k} (x)  \leq \tau (x)$. As a consequence, we have shown that, with probability at least $1-\delta$, 
   $$ \forall x\in S_X,\qquad  \hat \tau_{n,k}(x) ^d \leq \frac{ 2 k  }{ n \ell(x)  }.$$
The result then follows from applying Theorem \ref{th:general}. The variance term is obtained just noting that $n P_n^X (\mathcal V (x) ) = k$ and $v = d+1 $ because the local map is valued in the collection of balls which VC dimension is given in \cite{wenocur1981some}. For the bias we use the Lipschitz condition and the inequality above since the \(\ell^2\)-diameter is twice the radius \(\hat \tau_{n,k}(x)\), which gives the upper bound with probability at least \(1 - 3\delta\).

\subsection*{Proof of Corollary \ref{corknn}}

By assumption, there is $(a_-,a_+) $ such that $ 0< a _- \leq 1\leq   a _+ < +\infty $ and $$  k \, a_-  \leq n^{2/(d+2)} \log((n+1)^{d+1}/\delta)^{d/(d+2)} \leq a_ + \, k.$$
When $n$ is large enough, $k$ satisfies $ 8 \log( 4 (2n+1)^{(d+1)} / \delta) \leq k \leq T_0^d n \ell(x) / 2.$
According to Theorem \ref{th:NN}, we have the following inequalities with probability at least $1-3\delta$, for all $x\in S_X$, \begin{align*} | \hat g_{\mathcal V}(x) - g(x)| \quad &\leq \quad \sqrt{\frac{2\sigma^2\log ( {(n+1)^{d+1}/\delta })  a_+ }{ n^{2/(d+2)} \log((n+1)^{d+1}/\delta)^{d/(d+2)} }} \\ &\qquad \quad + \quad  2 L(\mathcal{V}(x)) \left( \frac{2 n^{2/(d+2)} \log((n+1)^{d+1}/\delta)^{d/(d+2)} }{n  \ell(x) a_- } \right)^{1/d} \\ &\leq \quad c \left( \dfrac{\log((n+1)^{d+1}/\delta)}{n}\right)^{1/(d+2)}  \dfrac{\sqrt{a_+}}{a_-}\end{align*} where $c = \sqrt{2\sigma^2} + 2 L(\mathcal{V}(x)) \left(2/\ell(x) \right)^{1/d}$.   Moreover, if $\ell$ is bounded below uniformly on $S_X$ by $b > 0$, we have $c\leq \sqrt{2\sigma^2} + 2 L \left(2/b\right)^{1/d}$. \qed

\section{Proof of the results stated in Section \ref{sec_wag}.}

\subsection*{Proof of Proposition \ref{propWager}}

Let $\delta \in (0,1/3)$. For each $x \in S_X$, let $\mathcal{V}_k(x)$ be the unique cell of depth $k$ containing $x$ and define its two childs $\mathcal{V}_{k+1}(x)$ and $\mathcal{V}_k(x) \backslash \mathcal{V}_{k+1}(x)$. We introduce a regularity constant $\tilde{\gamma} = {\alpha b}/{(8M)}$. According to  Lemma \ref{lemmaWager_new}, under the condition that $16 \log( 4 (2n+1)^{2d} / \delta) \leq n \alpha^{N}$, we have with probability at least $1-2\delta$, for all $x\in [0,1]^d $ and $k \in  \{0,\dots,N-1 \}$, 
$$\tilde{\gamma} \leq \dfrac{\lambda(\mathcal{V}_{k+1}(x))}{\lambda(\mathcal{V}_{k}(x))} \quad \text{and } \quad \tilde{\gamma} \leq \dfrac{\lambda(\mathcal{V}_k(x) \backslash \mathcal{V}_{k+1}(x))}{\lambda(\mathcal{V}_{k}(x))}.$$
At step $k$, suppose the split occurs along the axis $j = D_k(x)$. Since the cell $\mathcal{V}_k(x)$ is a hyper-rectangle, its volume is given by $\lambda(\mathcal{V}_k(x)) = \prod_{i=1}^d h_i(\mathcal{V}_k(x))$, where $h_i(\mathcal{V}_k(x))$ is the length of the $i$-th side at depth $k$ of $\mathcal{V}_k(x)$. As only the $j$-th side is affected by the split, we have  that
$$\dfrac{\lambda(\mathcal{V}_{k+1}(x))}{\lambda(\mathcal{V}_{k}(x))} = \frac{h_j(\mathcal{V}_{k+1}(x))}{h_j(\mathcal{V}_k(x))}$$ and $$\dfrac{\lambda(\mathcal{V}_k(x) \backslash \mathcal{V}_{k+1}(x))} {\lambda(\mathcal{V}_{k}(x))} = \frac{h_j(\mathcal{V}_k(x)) - h_j(\mathcal{V}_{k+1}(x))}{h_j(\mathcal{V}_k(x))} = 1 - \frac{h_j(\mathcal{V}_{k+1}(x))}{h_j(\mathcal{V}_k(x))}.$$  
When $j\neq D_k(x)$, we simply have $ h_j(\mathcal{V}_{k+1}(x)) = h_j(\mathcal{V}_k(x))$. It follows, with probability at least $1-2\delta$, for all $x\in [0,1]^d $ and $k \in  \{0,\dots,N-1 \}$, 
$$\tilde{\gamma} \mathds 1_{j=D_k(x)} + \mathds 1_{j\neq D_k(x)}\leq \frac{h_j(\mathcal{V}_{k+1}(x))}{h_j(\mathcal{V}_k(x))} \leq (1 - \tilde{\gamma})\mathds 1_{j=D_k(x)} + \mathds 1_{j\neq D_k(x)}.$$
The upper bound $h_j(\mathcal{V}_{k+1}(x)) \leq (1 - \tilde{\gamma}) h_j(\mathcal{V}_k(x))$ ensures that the side length strictly decreases by at least a factor $(1-\tilde{\gamma})$ at each split along axis $j$. By induction, if $N_j$ denotes the number of splits performed along the $j$-th coordinate to reach depth $N$, and since the initial side length is $h_j(\mathcal{V}_0(x)) =1$, then, as soon as $16\log( 4 (2n+1)^{2d} / \delta) \leq n \alpha^{N}$, we have with probability at least $1-2\delta$, for all $x\in [0,1]^d$ and $j \in \{1,\dots,d\}$, 
$$h_j(\mathcal{V}_N(x)) \leq (1 - \tilde{\gamma})^{N_j}.$$
Specifically, as soon as $16 \log( 4(2n+1)^{2d} / \delta) \leq n \alpha^{N}$, the maximum side length $h_+$ satisfies with probability at least $1-2\delta$,  for all $x\in [0,1]^d$, 
$$h_{+}(\mathcal{V}_N(x)) \leq (1 - \tilde{\gamma})^{\min_{j} N_j}.$$ 
We denote by $W$ this event. The sequence of chosen axes $(D_i(x))_{i=1,\dots,N}$ is i.i.d. with $\PP(D_i(x)=j)=p_j \ge \pi/d$. In particular, for each axis $j$, the number of splits $N_j$ after $N$ steps is a sum of Bernoulli variables whose expectation satisfies $\mathbb{E}[N_j]=N p_j \ge N\pi/d$. By the multiplicative Chernoff bound, if $Z$ is a sum of independent Bernoulli variables with mean $\mu$, then for $0 < t < 1$:
$$\PP\big(Z \le (1- t)\mu\big) \le \exp\!\left(-\frac{t^2}{2}\mu\right).$$ 
Then, for each coordinate $j \in \{1, \dots, d\}$, we have 
$$\PP\left( N_j \le (1-t) N \frac{\pi}{d}\right) \leq \PP\Big( N_j \le (1-t) \mathbb{E}[N_j] \Big) \le \exp\!\left(-\frac{t^2}{2}N p_j\right)\le \exp\!\left(-\frac{t^2}{2}  \frac{N\pi}{d}\right) .$$
Applying the union bound over the $d$ coordinates, we define the event $\bar{B} $, the complement of $B$, such that 
$$\PP(\bar{B}) = \PP\left( \min_{1\le j\le d} N_j \le (1-t) \frac{N\pi}{d}\right) \leq d \exp\!\left(-\frac{t^2}{2}  \frac{N\pi}{d}\right) .$$
Thus, with probability at least $1 - d\exp\!(-{t^2N\pi }/{(2d)} ) = 1-\eta$, we have
$$\min_{1\le j\le d} N_j \ge (1-t)  \frac{N\pi}{d} =  \frac{N\pi}{d}  - \sqrt{\frac{2N\pi}{d} \log(d/\eta)},$$
by setting $t = \sqrt{2d\log(d/\eta)/(N\pi)}$. Choosing $\eta = \delta$, we find that $W \cap B$ occurs with probability at least $1-3\delta$. That is, as soon as $16\log( 4 (2n+1)^{2d} / \delta) \leq n \alpha^{N}$, we have with probability at least $1-3\delta$, for all $x\in [0,1]^d$,
$$h_{+}(\mathcal{V}_N(x)) \le (1 - \tilde{\gamma})^{\min_j N_j} \le (1 - \tilde{\gamma})^{N \pi/d - \sqrt{2N\pi \log(d/\delta)/d}}.$$
Finally, under the assumption $N \pi \geq 8d \log(d/\delta)$, we obtain with probability at least $1-3\delta$, for all $x\in [0,1]^d$, 
$$h_{+}(\mathcal{V}_N(x))\le (1 - \tilde{\gamma})^{N \pi/(2d)}.$$ \qed

\subsection*{Proof of Theorem \ref{Wager}}

Let $\delta \in (0,1/5)$. We apply Theorem \ref{th:general} to the case of trees (where the general bound holds with probability at least $1-2\delta$). In this setting, the Vapnik-Chervonenkis dimension is $v = 2d$. At depth $N$, the cell is denoted by $\mathcal{V}(x) = \mathcal{V}_N(x)$. By construction, we have $P_n^X(\mathcal{V}_N(x)) \geq \alpha^N$, and the diameter is bounded by $\text{diam}(\mathcal{V}_N(x)) \leq \sqrt{d} h_+(\mathcal{V}_N(x))$, where $h_+(\mathcal{V}_N(x))$ denotes the maximum side length among all $d$ dimensions of the cell $\mathcal{V}_N(x)$ at depth $N$.
Then, with probability at least $1 - 2 \delta $, for all $x\in S_X$,
   % \begin{align*}
   %    &|  \hat g_{\mathcal V}(x) - g(x)| \\
   %    &\leq \sqrt{\frac{ 2 \sigma^2 \log\left( \frac{ (n+1) ^v }%{\delta} \right)}{ n \PP_n(\mathcal V (x))  } } + \sup_{y \in %\mathcal V (x)} |g(y) - g(x)|
    %\end{align*}
        \begin{align*}
       |  \hat g_{\mathcal V}(x) - g(x)| \leq \sqrt{\frac{ 2 \sigma^2 \log\left( \frac{ (n+1) ^{2d} }{\delta} \right)}{ n \alpha^N  } } + L(\mathcal V(x)  ) \, \sqrt{d} \, h_+(\mathcal{V}_N(x)).
    \end{align*}
According to Proposition \ref{propWager}, under the conditions $N \pi \geq 8d \log(d/\delta)$ and $16 \log ( 4 (2n+1)^{2d} / \delta) \leq n \alpha^{N}$, we obtain with probability at least $1-3\delta$, for all $x\in [0,1]^d$, 
$$h_{+}(\mathcal{V}_N(x))\le (1 - \tilde{\gamma})^{N \pi/(2d)}.$$ 
By combining these results, we conclude that the stated theorem holds with probability at least $1 - 5\delta$. \qed

\subsection*{Proof of Theorem \ref{Wager_subopt}}
The proof follows from an application of Proposition \ref{lemmaPZ}, established in Section \ref{sec_sub}, a preliminary section dedicated to the study of blind tree constructions. 

Let $x \in S_X$ and denote by $\mathcal{V}_N(x)$ the cell containing $x$ at step $N$. Let $h_+(\mathcal{V}_N(x))$ and $h_-(\mathcal{V}_N(x))$ denote the maximum and minimum side lengths of the cell $\mathcal{V}_N(x)$, respectively. By using the bounds on $(U_i)_{i=1,\dots,N}$, we have $$ \dfrac{1}{12d} \dfrac{\min_{1 \leq i \leq N} \EE(E_i^2)^2}{\max_{1 \leq i \leq N} \mathbb{E}(E_i^4)} \geq \frac{1}{12d} \left( \frac{\log(1-\rho)}{\log(\rho)} \right)^4,$$ and thus by Proposition \ref{lemmaPZ} 
    \begin{align*}
    \PP\left(\frac{h_+(\mathcal V _ N (x) )}{h_-(\mathcal V_N  (x) )} 
\geq (1-\rho)^{-\sqrt{N/d}} \right) &= 
\PP\left(\frac{h_+(\mathcal V _ N (x) )}{h_-(\mathcal V_N  (x) )} 
\geq \exp\left(\sqrt{\frac{N}{d} {\log(1-\rho)^2}}\right)\right) \\ &\geq \PP\left(\frac{h_+(\mathcal V _ N (x) )}{h_-(\mathcal V_N  (x) )} 
\geq \exp\left(\sqrt{\frac{N}{d} \min_{1 \leq i \leq N} \mathbb{E}(E_i^2)}\right)\right) \\ &\geq  \dfrac{1}{12d} \dfrac{\min_{1 \leq i \leq N} \EE(E_i^2)^2}{\max_{1 \leq i \leq N} \mathbb{E}(E_i^4)} \geq \frac{1}{12d} \left( \frac{\log(1-\rho)}{\log(\rho)} \right)^4.    
\end{align*} 

Furthermore, according to Lemma \ref{reci_wager_new}, by taking $\delta = \left( {\log(1-\rho)}/{\log(\rho)} \right)^4 /(48d)$, we obtain that the tree is, with probability at least $1-2\delta$, $\tilde{\alpha}$-regular with $\tilde{\alpha} = {b\rho}/{(8M)}$ as soon as $16\log( 4 (2n+1)^{2d} / \delta) \leq n (b\rho / M ) ^{N} $. Note that we indeed have $\delta \in (0,1/2)$ since for all $\rho \in (0,1/2], 0 < \delta < (48d)^{-1} < 1/2.$

Thus, with probability at least $\left( \log(1-\rho)/\log(\rho) \right)^4 /(12d) - 2\delta = \left( \log(1-\rho)/\log(\rho) \right)^4 /(24d)$, we have an $\tilde{\alpha}$-regular tree such that the shape regularity factor ${h_+(\mathcal V _ N (x) )}/{h_-(\mathcal V_N (x) )}$ is bounded from below by $(1-\rho)^{-\sqrt{N/d}}.$ Note that this factor diverges as $N \to \infty$. \qed

\subsection*{Proof of Theorem \ref{th_cart_like1}}

The proof follows from a straightforward application of the next result, which is stated for general local regression maps. 

\begin{theorem}\label{th_cart_like}
Let $S_X=[0,1]^d$, $\delta \in (0,1/3) $, $n\geq 1$, $d\geq 1$, and $m\geq 4\log(4(2n+1)^{2d}/\delta )$. Suppose that \ref{cond:D}, \ref{cond:density_XCART}, \ref{cond:epsilon} and \ref{cond:reg4} are fulfilled. Let $\beta \geq 2$ and suppose that $\mathcal V$ is a local regression map valued in the set of hyper-rectangles contained in $S_X$, for all $V\in \left\{ \mathcal V(x) \, :\, x\in \mathbb R^d \right\}$,
    \begin{align*}
        h_+(V) \leq \beta h_- (V)
\quad \text{and}\quad 
       n P_n^X( V) \geq  m,
    \end{align*}
    then we have, with probability at least $1-3\delta$, for all $x \in S_X$,
    \begin{align*} 
            |  \hat g_\mathcal V (x) - g(x)  | \leq \sqrt { \frac{2 \sigma ^2 \log((n+1) ^{2d}/\delta) } {  m}} + L(\mathcal{V}(x))  \beta \sqrt{d} \left(\frac{5 m}{ n b}\right)^{1/d}.
    \end{align*}
    
\end{theorem}

    Note that,  when growing the tree, the constraint $h_+(V) \leq \beta h_- (V) $ can never be a stopping criterion because one can always select the largest side and split it in the middle. When the tree is fully grown according to the prescribed rules,  acceptable splits are no longer possible. Therefore any $V $ satisfies 
    $$ 2 m \geq   nP_n^X(V) \geq  m.$$
    %This must be satisfied if all splits along $k$ are rejected. 
Since the Vapnik dimension of hyper-rectangles is $ v = 2d$, using Assumption \ref{cond:density_XCART} and Theorem \ref{th:vapnik_normalized}, then for all $\delta \in (0,1)$ and $m\geq 4\log(4(2n+1)^{2d}/\delta )$, we obtain with probability at least $1 - \delta$,
$$b h_-^d \leq P^X(V) \leq \dfrac{4}{n} \log\left(\dfrac{4(2n+1)^{2d}}{\delta} \right) + 2 P_n^X(V) \leq \frac{m}{n} + \frac{4m}{n} = \frac{5m}{n}.$$
 In addition,
$$\diam(V) \leq \sqrt{d} h_+ \leq \sqrt{d} \beta h_- \leq  \sqrt{d} \beta \left(\dfrac{5 m}{n b}\right)^{1/d}.$$
It remains to apply Theorem \ref{th:general} and to use that $n P_n^X(V) \geq m$ for the variance term to get the stated result.
\qed

\subsection*{Proof of Corollary \ref{cor_cart_like1}}
We apply Theorem \ref{th_cart_like1} to the stipulated choice of $m$. By assumption, there exist $(a_-, a_+)$ such that $0 < a_- \leq 1 \leq a_+ < +\infty$ and$$m \, a_- \leq n^{2/(d+2)} \log((n+1)^{2d}/\delta)^{d/(d+2)} \leq a_+ \, m.$$When $n$ is large enough, the choice of $m \asymp n^{2/(d+2)} \log((n+1)^{2d}/\delta)^{d/(d+2)}$ ensures that $m$ satisfies the condition $m \geq 4 \log(4(2n+1)^{2d}/\delta)$ required by Theorem \ref{th_cart_like1}.  According to Theorem~\ref{th_cart_like1}, we have the following inequalities with probability at least $1-3\delta$, for all $x \in S_X$, 
\begin{align*}
| \hat g_{\mathcal V}(x) - g(x)| &\leq \sqrt{\frac{2 \sigma^2 \log((n+1)^{2d}/\delta) a_+}{n^{2/(d+2)} \log((n+1)^{2d}/\delta)^{d/(d+2)}}} \\ &\qquad + L(\mathcal{V}(x)) \beta \sqrt{d} \left( \frac{5 n^{2/(d+2)} \log((n+1)^{2d}/\delta)^{d/(d+2)}}{n b a_-} \right)^{1/d} \\ &\leq \left[ \sqrt{2 \sigma^2 a_+} + L(\mathcal{V}(x))  \beta \sqrt{d} \left( \frac{5}{b a_-} \right)^{1/d} \right] \left( \frac{\log((n+1)^{2d}/\delta)}{n} \right)^{1/(d+2)} \\ &\leq c \left( \dfrac{\log((n+1)^{2d}/\delta)}{n}\right)^{1/(d+2)} \dfrac{\sqrt{a_+}}{a_-}
\end{align*}
where $c = \sqrt{2\sigma^2} + L(\mathcal{V}(x))  \beta  \sqrt{d} (5/b)^{1/d}$. By upper bounding $L(\mathcal{V}(x))$ by $L$, the inequality becomes uniform over $x \in S_X$. Taking the supremum then yields the desired result. \qed

\subsection*{Proof of Theorem \ref{CART_pratique}}
For any $X_i \in \mathcal{V}(x)$, where $i \in \{1, \dots, n\}$, it holds by definition that for each dimension $j \in \{1, \dots, d\}$, $|X_{i,j} - x_j| \leq h_j(\mathcal{V}(x))$ where $h_j(\mathcal{V}(x)) = \sup \{ |u_j - v_j| : u, v \in \mathcal{V}(x) \}$ denotes the side length of the cell along the $j$-th dimension. Under the directional Lipschitz assumption \ref{cond:lip_direc}, the bias term can then be bounded as follows
\begin{align*}\left| \frac{\sum_{i=1}^n (g(X_i) - g(x)) \mathds{1}_{\mathcal{V}(x)}(X_i)}{\sum_{i=1}^n \mathds{1}_{\mathcal{V}(x)}(X_i)} \right| &\leq \frac{\sum_{i=1}^n |g(X_i) - g(x)| \mathds{1}_{\mathcal{V}(x)}(X_i)}{\sum_{i=1}^n \mathds{1}_{\mathcal{V}(x)}(X_i)} \\
&\leq \frac{\sum_{i=1}^n \left( \sum_{j=1}^d L_j h_j(\mathcal{V}(x)) \right) \mathds{1}_{\mathcal{V}(x)}(X_i)}{\sum_{i=1}^n \mathds{1}_{\mathcal{V}(x)}(X_i)} \\
&= \sum_{j=1}^d L_j h_j(\mathcal{V}(x)).\end{align*}
The stated theorem then follows directly from the bias-variance decomposition established in Theorem \ref{th:general}. \qed

\section{Blind tree constructions (a preliminary study to the proof of Theorem \ref{Wager_subopt})} \label{sec_sub}

The proof of Theorem \ref{Wager_subopt} requires some development about blind tree constructions. These tree are characterized by the independence of the split direction and position.

Let $h_{k}(V)$ denote the length of the $k$-th side of a cell $V$. The tree is constructed recursively as follows: at each step $i$, for each terminal leaf $V$, an axis $D_i$ is drawn uniformly from $\{1, \dots, d\}$ and a split position $S_i$ is drawn from a distribution on $(0, 1)$. The cell $V$ is then partitioned along coordinate $k = D_i$ into two daughter cells with respective side lengths $h_{k}(V) S_i$ and $h_{k}(V)(1-S_i)$. Let $\mathcal{V}_i(x)$ denote the cell containing a given point $x$ at step $i$. We define the relative child-to-parent side length ratio as $U_i = h_{D_i}(\mathcal{V}_i(x)) / h_{D_i}(\mathcal{V}_{i-1}(x))$, noting that $S_i$ corresponds to either $U_i$ or $1 - U_i$ depending on which side of the split $x$ falls. Throughout the following, let $E_i = -\log(U_i)$. We assume that the sequence $(D_i)_{i \geq 1}$ is i.i.d., drawn uniformly from $\{1, \dots, d\}$, and that $(U_i)_{i \geq 1}$ is a sequence of random variables independent of $(D_i)_{i \geq 1}$. We first establish a lemma that provides a sufficient condition for the lack of shape regularity in such trees. This result utilizes the Paley–Zygmund inequality to show that the aspect ratio of the cells remains large with positive probability.

\begin{proposition}\label{lemmaPZ}
For any depth $N \geq 1$ and all $x \in S_X$, let $h_+(\mathcal{V}_N(x))$ and $h_-(\mathcal{V}_N(x))$ denote the maximum and minimum side lengths of the cell $\mathcal{V}_N(x)$, respectively. Then,
    \begin{align*}
\PP\left(\frac{h_+(\mathcal V _ N (x) )}{h_-(\mathcal V_N  (x) )} 
\geq \exp\left(\sqrt{\frac{N}{d} \min_{1 \leq i \leq N} \mathbb{E}(E_i^2)}\right)\right) &\geq \dfrac{1}{12d} \dfrac{\min_{1 \leq i \leq N} \EE(E_i ^2)^2}{\max_{1 \leq i \leq N} \mathbb{E}(E_i ^4) }.    
\end{align*} 
\end{proposition}

The proof is given at the end of the section. Thus, when the second and fourth moments of $(E_i)_{i=1, \dots N}$ are uniformly bounded in $N$ -- this means that, approximately, the $(U_i)_{i=1, \dots, N}$ are on average far from 0 and 1 as $N$ grows large -- the associated tree is not shape regular, which is the subject of the next corollary.

\begin{corollary}\label{cor_moments}
Let $x \in S_X$. Suppose there exist constants $c_1 > 0$, $c_2 > 0$ such that for all $N \geq 1$, $\min_{1 \leq i \leq N} \EE(E_i^2) \geq c_1$ and $\max_{1 \leq i \leq N} \mathbb{E}(E_i^4) \leq c_2$, then with strictly positive probability, the ratio $h_+(\mathcal V_N (x) ) / h_-(\mathcal V_N (x) )$ is bounded below by $\exp(\sqrt{N c_1 /{d}})$. Therefore, the associated tree is not shape regular.
\end{corollary}
\begin{proof}
We apply the previous lemma \ref{lemmaPZ} by bounding the moments from below and above using the assumptions of the proposition. We then obtain
    \begin{align*}
\PP\left(\frac{h_+(\mathcal V _ N (x) )}{h_-(\mathcal V_N  (x) )} 
\geq \exp\left(\sqrt{\frac{N}{d} c_1 } \right)\right) &\geq \dfrac{1}{12d} \dfrac{c_1^2}{c_2}.    
\end{align*}    
\end{proof}

We also have the following corollary, which is of interest in the case of purely random trees.

\begin{corollary}\label{cor_PRT}
Let $x \in S_X$. If the random variables $(S_i)_{i=1,\dots,N}$ representing cut sizes follow the same distribution, symmetric around $1/2$ and satisfy $\mathbb{E}(\log(S_1)^4) < + \infty$, then with strictly positive probability, the ratio $h_+(\mathcal{V}_N(x)) / h_-(\mathcal{V}_N(x))$ is bounded below by $\exp\left( \sqrt{N \sqrt{\mathbb{E}(\log(S_1)^2)} / d} \right)$. Hence, the associated tree is not shape regular.
\end{corollary}
\begin{proof}
    By assumption on the random variables $(S_i)_i$, the $(U_i)_i$ share the same distribution, so the minimum and maximum of the moments of the $(E_i)_i$ are constant.
\end{proof}

This geometric divergence confirms that purely random splitting rules lack the necessary adaptivity to balance the cell diameters across all dimensions. As established in Proposition \ref{contre_ex}, such an unbalanced structure is insufficient to capture the local variations of Lipschitz functions at the minimax optimal rate. This highlights the importance of having the cutting directions or cut sizes evolve over time to take into account the geometry of the cell. This is the case in Mondiran's tree where the probability to split according to a direction is proportional to its length.

\subsubsection*{Proof of Proposition \ref{lemmaPZ}}
For a given leaf, after $N$ stages,  the $k$-th length has the following representation 
$$ h _ k(\mathcal V (x) ) = U_1^{B_1^{(k)}}\times \ldots \times U_N^{B_N^{(k)}} = \exp\left( \sum_{i=1} ^ N B_i^{(k)} \log(U_i) \right)  $$
where $B_i^{(k)}  = \ind_{D_i = k }$. It follows that
\begin{align*}
 &h_+(\mathcal V_N (x) ) = \exp\left( \max_{k=1,\ldots, d}  \sum_{i=1} ^ N B_i^{(k)} \log(U_i) \right), \\
 &h_-(\mathcal V_N (x) ) =  \exp\left(\min_{k=1,\ldots, d}  \sum_{i=1} ^ N B_i^{(k)} \log(U_i) \right), 
\end{align*}

and the expression of the ratio is
\begin{align*}
 h_+(\mathcal V_N (x) ) / h_-(\mathcal V_N (x) )  &= \exp \left(   \max_{1 \leq k,j \leq d} \sum_{i=1} ^N (B_i^{(k)} - B_i^{(j)} ) E_i \right) \end{align*}
where $E_i = - \log(U_i)$.

By denoting $V_i^{k,j} = B_i^{(k)} - B_i^{(j)}$, we get
\[V_i^{k,j} = \begin{cases} 
1 & \text{with probability } 1/d \\
0 & \text{with probability } 1 - 2/d \\
-1 & \text{with probability } 1/d 
\end{cases}.\]
Note that the variables \((V_i^{k,j})_{i=1,\dots,N}\) are mutually independent because the \((D_i)_{i=1,\dots,N}\) are independent. Furthermore, since the \(U_i\)'s are independent of the \(V_i\)'s, the \(V_i^{k,j}\)'s are independent of the \(E_i\)'s. Let \(Z_{k,j} = \sum_{i=1}^{N} V_i^{k,j} E_i\) such that
\[
\frac{h_+(\mathcal V_N (x) )}{h_-(\mathcal V_N (x) )} = \exp \left( \max_{1 \leq k,j \leq d} Z_{k,j} \right).
\]
Note that \(Z_{k,j} = - Z_{j,k}\) and \(Z_{k,k} = 0\), which gives
\[
\max_{1 \leq k,j \leq d} Z_{k,j} = \max_{1 \leq k < j \leq d} | Z_{k,j} |
\]
and thus the formula 
\[
\frac{h_+(\mathcal V_N (x) )}{h_-(\mathcal V_N (x) )} = \exp \left( \max_{1 \leq k < j \leq d} | Z_{k,j} | \right).
\]
By using the Paley-Zygmund inequality to $Z_{k,j}^2$, we get for all $\theta \in (0,1)$,
\[
\PP\left(|Z_{k,j}| \geq \sqrt{\theta} \sqrt{\mathbb{E}(Z_{k,j}^2)} \right) \geq (1-\theta)^2 \frac{\mathbb{E}(Z_{k,j}^2)^2}{\mathbb{E}(Z_{k,j}^4)}.
\]
We therefore seek to calculate the 2nd and 4th moments of \(Z_{k,j}\). \\ Since \(Z_{k,j}^2 = \sum_{i \neq \ell} V_i^{k,j} V_\ell^{k,j} E_i E_\ell + \sum_{i=1}^N V_i^{k,j \, 2} E_i^2 \), \(\mathbb{E}(V_i^{k,j}) = 0\) and by independence of the \((V_i^{k,j})_i\) from each other and from the \((E_i)_i\), we obtain 
$$\mathbb E(Z_{k,j}^2) = \sum_{i=1}^N \mathbb{E}((V_i^{k,j})^2) \mathbb{E}(E_i^2) = \frac{2}{d} \sum_{i=1}^N \mathbb{E}(E_i^2) \geq \frac{2N a(N)}{d}$$
where $a(N) = \min_{1 \leq i \leq N} \EE(E_i^2).$ 

Moreover, $(V_i^{k,j})_i$ are independent of each other and of $(E_i)_i$, with $\mathbb{E}[V_i^{k,j}] = \mathbb{E}[(V_i^{k,j})^3] = 0$, thus, we obtain
\begin{eqnarray*}
 \mathbb{E}(Z_{k,j}^4) &=& \mathbb{E} \left[ \left( \sum_{i=1}^N V_i^{k,j} E_i \right)^4 \right]
\\ &=& \sum_{i=1}^N \mathbb{E}\left[ \left(V_i^{k,j}\right)^4 \right]\, \mathbb{E}[E_i^4]
+ 6 \sum_{1 \leq i < j \leq N} \mathbb{E}\left[ \left(V_i^{k,j}\right)^2 \right]\, \mathbb{E}\left[ \left(V_j^{k,j}\right)^2 \right]\, \mathbb{E}[E_i^2 E_j^2] \\ &=&   \frac{2}{d} \sum_{i=1}^N \mathbb{E}[E_i^4]
+ 6 \left(\frac{2}{d}\right)^2 \sum_{1 \leq i < j \leq N} \mathbb{E}[E_i^2 E_j^2] \\ &\leq& \frac{2Nb(N)}{d} +   3 \left(\frac{2}{d}\right)^2 N(N-1) c(N)
\end{eqnarray*}
where $b(N) = \max_{1 \leq i \leq N} \mathbb{E}[E_i^4] $ and $c(N) = \max_{1 \leq i < j \leq N} \mathbb{E}[E_i^2 E_j^2]$. By Cauchy-Schwarz inequality, $c(N) \leq b(N).$ Then, $$ \mathbb{E}(Z_{k,j}^4) \leq \frac{2Nb(N)}{d} \left( 1 + \frac{6}{d} (N-1) \right) \leq \frac{2Nb(N)}{d}  \times 6N.$$

Let $\epsilon $ such that $$(1-\theta)^2 \frac{\mathbb{E}(Z_{k,j}^2)^2}{\mathbb{E}(Z_{k,j}^4)} = \frac{1}{\epsilon^2}$$ i.e. $$\theta = 1 - \frac{\sqrt{\mathbb{E}(Z_{k,j}^4)}}{\epsilon \, \mathbb{E}(Z_{k,j}^2) }.$$

By using the Paley-Zygmund inequality to $Z_{k,j}^2$, 
\[
\PP\left(|Z_{k,j}| \geq  \sqrt{\mathbb{E}(Z_{k,j}^2) - \frac{1}{\epsilon } {\sqrt{\mathbb{E}(Z_{k,j}^4)}}} \right) \geq \frac{1}{\epsilon^2}.
\]

Then, 
\[
\PP\left(|Z_{k,j}| \geq  \sqrt{\frac{2Na(N)}{d} - \frac{1}{\epsilon } {\sqrt{\frac{2Nb(N)}{d}  \times 6N}} }\right) \geq \frac{1}{\epsilon^2}
\]
and \[
\PP\left(|Z_{k,j}| \geq  \sqrt{\frac{2N}{d} ({a(N)} - \frac{1}{\epsilon } {\sqrt{3b(N)d}}} ) \right) \geq \frac{1}{\epsilon^2}.
\]
Let $\epsilon = 2\sqrt{3b(N)d}/a(N)$ i.e. $\epsilon^2 = 12d b(N)/a(N)^2$, then
 \[
\PP\left(|Z_{k,j}| \geq  \sqrt{\frac{Na(N)}{d} } \right) \geq \dfrac{a(N)^2}{12d b(N)}.
\]

Finally, by the following lower bound,
\[
\frac{h_+(\mathcal V_N (x) )}{h_-(\mathcal V_N (x) )} = \exp \left( \max_{1\leq k < j \leq d } |Z_{k,j}| \right) \geq \exp (|Z_{1,2}|),
\]
we get, for any $N \geq 1$, 
\begin{align*}
\PP\left(\frac{h_+(\mathcal V_N (x) )}{h_-(\mathcal V_N (x) )} 
\geq \exp\left(\sqrt{\frac{Na(N)}{d}}\right)\right) &\geq \PP\left(\exp(|Z_{1,2}|) \geq \exp\left(\sqrt{\frac{Na(N)}{d}}\right)\right) \\ 
&= \PP\left(|Z_{1,2}| \geq \sqrt{\frac{Na(N)}{d}}\right) \geq  \dfrac{a(N)^2}{12d b(N)} \\ &= \dfrac{1}{12d} \dfrac{\min_{1 \leq i \leq N} \EE(E_i^2)^2}{\max_{1 \leq i \leq N} \mathbb{E}(E_i^4) }.    
\end{align*} \qed

% \subsubsection*{Proof of Corollary \ref{cor_moments}}
% We apply the previous lemma \ref{lemmaPZ} by bounding the moments from below and above using the assumptions of the proposition. We then obtain
%     \begin{align*}
% \PP\left(\frac{h_+(\mathcal V _ N (x) )}{h_-(\mathcal V_N  (x) )} 
% \geq \exp\left(\sqrt{\frac{N}{d} c_1 } \right)\right) &\geq \dfrac{1}{12d} \dfrac{c_1^2}{c_2}.    
% \end{align*} \qed

% \subsubsection*{Proof of Corollary \ref{cor_PRT}}
% By assumption on the random variables $(S_i)_i$, the $(U_i)_i$ share the same distribution, so the minimum and maximum of the moments of the $(E_i)_i$ are constant. \qed

\section{Auxiliary results and technical lemmas}

Let us start with a result establishing some conditional independence property under \ref{cond:D}.

%\begin{align*}
%\mathbb E [ f(\epsilon_1) g(\epsilon_2) | X_1,X_2] &=\mathbb E [ \mathbb E [ f(\epsilon_1) | X_1,X_2,\epsilon_2] g(\epsilon_2) | X_1,X_2] \\
%&= \mathbb E [ \mathbb E [ f(\epsilon_1) | X_1] g(\epsilon_2) | X_1,X_2] \\
%&= \mathbb E [ f(\epsilon_1) | X_1] \mathbb E [  g(\epsilon_2) | X_2] .
%\end{align*}
%This implies that $\epsilon_1$  and $\epsilon _2$ are independent given $X_1,X_2$. So an intermediate lemma would be needed:
\begin{lemma}\label{lemmaA}
    Assume \ref{cond:D}. Then $(Y_i) _{i = 1,\ldots, n} $ is an independent collection of random variables, conditionally on $(X_i)_{i = 1,\ldots, n}$. 
\end{lemma}

\begin{proof}
   % Since $g(X_i)$ is $X_i$--measurable, it is sufficient to prove that $(Y_i = \varepsilon_i + g(X_i))_{i=1, \dots,n}$ is an independent collection of random variables conditionally on $(X_i)_{i=1,\dots,n}.$
    Let $(\phi_i)_{i=1, \dots,n}$ and $(\psi_i)_{i=1, \dots,n}$ be bounded and measurable functions. Then
    \begin{eqnarray*}
    \EE\left( \prod_{i=1}^n \EE\left(\psi_i(Y_i) \middle| X_i \right) \prod_{i=1}^n \phi_i(X_i)\right) &=& \EE\left(\prod_{i=1}^n \EE\left( \psi_i(Y_i) \phi_i (X_i)  | X_i\right)  \right) \\ &=& \prod_{i=1}^n \EE\left( \EE\left( \psi_i(Y_i) \phi_i (X_i) \middle| X_i \right) \right)
    \end{eqnarray*}
because $\EE\left( \psi_i(Y_i) \phi_i (X_i) \middle| X_i \right)$ is $X_i$--measurable and $(X_i)_{i=1,\dots,n}$ are independent. Hence,
    \begin{eqnarray*}
    \EE\left( \prod_{i=1}^n \EE\left(\psi_i(Y_i) \middle| X_i \right) \prod_{i=1}^n \phi_i(X_i)\right) &=& \prod_{i=1}^n \EE\left(\psi_i(Y_i) \phi_i(X_i)\right) \\ &=& \EE\left(\prod_{i=1}^n \psi_i(Y_i) \phi_i(X_i) \right)
    \end{eqnarray*}
by independence of $(X_i, Y_i)_{i=1,\dots,n}$. By definition of conditional expectation, we obtain 
    $$\EE\left(\prod_{i=1}^n \psi_i(Y_i) \middle| (X_j)_{j=1,\dots,n} \right) = \prod_{i=1}^n \EE\left(\psi_i(Y_i) \middle| X_i \right)$$
which means that $(Y_i)_{i=1,\dots,n}$ is an independent collection of random variables conditionally on $(X_i)_{i=1,\dots,n}.$
\end{proof}

The following lemma ensures that, under a critical mass condition and with bounded covariates density (both above and below), the randomness of the data cannot distort the structure of the tree, thereby forcing each split to reduce the volume by a deterministic factor.

\begin{lemma}\label{lemmaWager0}
Consider a tree of depth $N$ on $[0,1]^d$. For $k \in  \{0,1,\dots,N \}$, let $\mathcal{V}_k(x)$ be the unique cell of depth $k$ containing $x$. Assume that we have the $( P^X_n,\alpha)$-regularity condition: for all $x\in [0,1]^d$ and  $k \in \{1,\ldots,  N\} $, $$  P_n^X(\mathcal{V}_{k}(x))  \ge \alpha P_n^X(\mathcal{V}_{k-1}(x)).$$ 
Under the condition $16\log( 4 (2n+1)^{2d} / \delta) \leq n \alpha^{N}$, with probability at least $1 - 2\delta$,  we have, for all $x\in [0,1]^d$ and $k \in  \{0,\dots,N \}$, 
$$  \frac{1}{ 4}    P ^X (\mathcal{V}_{k}(x))  \leq  P^X_n(\mathcal{V}_{k}(x))\leq 2 P^X (\mathcal{V}_{k}(x)) . $$
\end{lemma}

\begin{proof}
For all $k\in \{1,\ldots, N\}$ and all $x\in [0,1]^d$, we have that $P_n^X(\mathcal V_k(x)) \geq \alpha P_n^X(\mathcal{V}_{k-1}(x))$. It follows that for all $k\in \{0,1,\ldots, N\}$ and all $x\in [0,1]^d$, $ P_n^X(\mathcal V_k(x))  \geq  \alpha^{k} P_n^X(\mathcal{V}_0(x)) = \alpha^{k}\geq \alpha^N $. As a consequence, our condition $16 \log( 4 (2n+1)^{2d} / \delta) \leq n \alpha^{N}$ implies that 
$$ \text{for all }x\in [0,1]^d \text{ and } k\in \{0,1,\ldots, N\} , \qquad \frac{ 4 \log( 4 (2n+1)^{2d} / \delta) }{ n P_n^X(\mathcal V_k(x)) } \leq \frac 1  4. $$ 
Using Vapnik's inequality (last statement in Theorem \ref{th:vapnik_normalized}) with the class $\mathcal{A}$ of hyper-rectangles in $\mathbb{R}^d$, we have $\mathcal{S}_{2n}(\mathcal{A}) \leq (2n+1)^{2d}$. As a consequence, with probability $1-\delta$, for all hyper-rectangles $V$,
\begin{align*}
    P^X(V) &\geq P_n^X (V) \left(1-\sqrt { \frac{4 \log( 4 \mathcal{S}_{2n}(\mathcal{A}) / \delta) }{ nP_n^X (V)} } \right) \\
    &\geq P_n^X (V) \left(1-\sqrt { \frac{4 \log( 4 (2n+1)^{2d}  / \delta) }{ nP_n^X (V)} } \right)\geq \frac 1 2 P_n^X (V). 
\end{align*}
The latter being valid for all hyper-rectangles, it must be true for $V = \mathcal{V}_{k}(x) $, for all $k = 0,1,\ldots, N$ and all $x\in [0,1]^d$, leading to, with probability $1-\delta$, for all  $x\in [0,1]^d$ and all $k\in \{ 0,\ldots, N\}$,
\begin{align}\label{event1_}
  P^X (\mathcal{V}_{k}(x) ) \geq \frac 1 2  P_n^X (\mathcal{V}_{k}(x) ),%\qquad   P^X (\mathcal{V}_k(x) \backslash \mathcal{V}_{k+1}(x)\} ) \geq \frac 1 2  P_n^X ( \mathcal{V}_k(x) \backslash \mathcal{V}_{k+1}(x) )    
\end{align}
% $$\min (P^X (\mathcal{V}_{k+1}(x) )  , P^X (\mathcal{V}_k(x) \backslash \mathcal{V}_{k+1}(x)\}  )  \geq \frac 1 2 \alpha P_n^X (\mathcal{V}_k(x))    .$$
Note that the above inequality implies that for all $x\in [0,1]^d$ and $k\in \{0,1,\ldots, N \}$,
$$ \frac{ 4 \log( 4 (2n+1)^{2d} / \delta) }{ n P^X(\mathcal V_k(x)) } \leq \frac{ 8 \log( 4 (2n+1)^{2d} / \delta) }{ n P^X_n(\mathcal V_k(x)) }\leq \frac 1  2. $$ 
In a similar way as before, we now apply another Vapnik's inequality (first statement in Theorem \ref{th:vapnik_normalized}) to obtain that, with probability at least $1 - \delta$, for  all $x\in [0,1]^d$ and all integer $k \in \{0,\dots,N\}$,
\begin{align}\label{event2_} 
P^X_n(\mathcal{V}_{k}(x))  &\geq P ^X (\mathcal{V}_{k}(x)) \left(1-\sqrt { \frac{4 \log( 4 (2n+1)^{2d} / \delta) }{ nP ^X (\mathcal{V}_{k}(x))  } } \right) \geq \left(1-\sqrt { \frac{1}{ 2} } \right)   P ^X (\mathcal{V}_{k}(x)).
\end{align}
% and
% $$ P^X ( \mathcal{V}_k(x) \backslash \mathcal{V}_{k+1}(x)\})  \geq \frac 1 2 \alpha P^X_n(\mathcal{V}_{k}(x)) \geq \frac 1 2 \left(1-\sqrt { \frac{1}{ 2} } \right)  \alpha P ^X (\mathcal{V}_{k}(x))   ,  $$
Noticing that $ (1- \sqrt {1 / 2} ) \geq 1/4$, combining both events \eqref{event1_}, \eqref{event2_} yields, with probability $1-2\delta$, for all $x\in [0,1]^d$ and  $k\in \{ 0,\ldots , N\}$,
$$P^X (\mathcal{V}_{k}(x)  \geq \frac 1 2  P^X_n(\mathcal{V}_{k}(x))\geq  \frac{1}{ 8}    P ^X (\mathcal{V}_{k}(x)) . $$
    
\end{proof}

\begin{lemma}\label{lemmaWager1}
Consider a tree of depth $N$ on $[0,1]^d$. For $k \in  \{0,1,\dots,N \}$, let $\mathcal{V}_k(x)$ be the unique cell of depth $k$ containing $x$. Assume that we have the $( P^X,\alpha)$-regularity condition: for all $x\in [0,1]^d$ and  $k \in \{1,\ldots,  N\} $, $$  P^X(\mathcal{V}_{k}(x))  \ge \alpha P^X(\mathcal{V}_{k-1}(x)).$$ 
Under the condition $16\log( 4 (2n+1)^{2d} / \delta) \leq n \alpha^{N}$, with probability at least $1 - 2\delta$,  we have, for all $x\in [0,1]^d$ and $k \in  \{0,\dots,N \}$, 
$$  \frac{1}{ 4}    P_n ^X (\mathcal{V}_{k}(x))  \leq  P^X(\mathcal{V}_{k}(x))\leq 2 P_n^X (\mathcal{V}_{k}(x)). $$
\end{lemma}

\begin{proof}
The proof is similar to that of Lemma \ref{lemmaWager0} putting  $P^X$ in place of $P_n^X$ and conversely. For all $k\in \{1,\ldots, N\}$ and all $x\in [0,1]^d$, we have that $P^X(\mathcal V_k(x)) \geq \alpha P ^X(\mathcal{V}_{k-1}(x))$. It follows that for all $k\in \{0,1,\ldots, N\}$ and all $x\in [0,1]^d$, $ P^X(\mathcal V_k(x))  \geq  \alpha^{k} P^X(\mathcal{V}_0(x)) = \alpha^{k}\geq \alpha^N $. As a consequence, our condition $16 \log( 4 (2n+1)^{2d} / \delta) \leq n \alpha^{N}$ implies that 
$$ \text{for all }x\in [0,1]^d \text{ and } k\in \{0,1,\ldots, N\} , \qquad \frac{ 4 \log( 4 (2n+1)^{2d} / \delta) }{ n P^X(\mathcal V_k(x)) } \leq \frac 1  4. $$ 
We now apply another Vapnik's inequality (first statement in Theorem \ref{th:vapnik_normalized}) to obtain that, with probability at least $1 - \delta$, for  all $x\in [0,1]^d$ and all integer $k \in \{0,\dots,N\}$,
\begin{align}\label{event2_0} 
P^X_n(\mathcal{V}_{k}(x))  &\geq P ^X (\mathcal{V}_{k}(x)) \left(1-\sqrt { \frac{4 \log( 4 (2n+1)^{2d} / \delta) }{ nP ^X (\mathcal{V}_{k}(x))  } } \right) \geq  \frac{1}{ 2}    P ^X (\mathcal{V}_{k}(x)).
\end{align}
Note that the above inequality implies that for all $x\in [0,1]^d$ and $k\in \{0,1,\ldots, N \}$,
$$ \frac{ 4 \log( 4 (2n+1)^{2d} / \delta) }{ n P^X_n(\mathcal V_k(x)) } \leq \frac{ 8 \log( 4 (2n+1)^{2d} / \delta) }{ n P^X(\mathcal V_k(x)) }\leq \frac 1  2. $$ 
Using Vapnik's inequality (last statement in Theorem \ref{th:vapnik_normalized}), we have, with probability $1-\delta$, for all $x\in [0,1]^d$ and $k\in \{0,1,\ldots, N \}$,
\begin{align}\label{event1_0}
    P^X(V)   &\geq P_n^X (V_k(x)) \left(1-\sqrt { \frac{4 \log( 4 (2n+1)^{2d}  / \delta) }{ nP_n^X (V_k(x))} } \right)\geq \left(1-\sqrt { \frac{1 }{2} } \right) P_n^X (V_k(x)). 
\end{align}
Noticing that $ (1- \sqrt {1 / 2} ) \geq 1/4$, combining both events \eqref{event2_0}, \eqref{event1_0} yields, with probability $1-2\delta$, for all $x\in [0,1]^d$ and  $k\in \{ 0,\ldots , N\}$,
$$P^X _n (\mathcal{V}_{k}(x))  \geq \frac 1 2  P^X(\mathcal{V}_{k}(x))\geq  \frac{1}{ 8}    P ^X_n (\mathcal{V}_{k}(x)) . $$

\end{proof}

\begin{lemma}\label{lemmaWager_new}
Consider a tree of depth $N$ on $[0,1]^d$. For $k \in  \{0,1,\dots,N \}$, let $\mathcal{V}_k(x)$ be the unique cell of depth $k$ containing $x$.
 Assume that we have the $(\alpha,P_n^X)$-regularity condition: for all $x\in [0,1]^d$ and  $k \in \{1,\ldots,  N\} $, 
$  P_n^X(\mathcal{V}_{k}(x))  \ge \alpha P_n^X(\mathcal{V}_{k-1}(x))$. Under the condition $16\log( 4 (2n+1)^{2d} / \delta) \leq n \alpha^{N}$, with probability at least $1 - 2\delta$,  we have, for all $x\in [0,1]^d$ and $k \in  \{1,\dots,N \}$, 
\begin{equation*}
 P^X(\mathcal{V}_{k}(x))  \ge \frac{\alpha}{8} P^X(\mathcal{V}_{k-1}(x)).    
\end{equation*}
Moreover, if $X$ admits a density $f_X$ bounded below by a constant $b > 0$ and above by a constant $M > 0$, then, with probability $1-2\delta$,  we have, for all $x\in [0,1]^d$ and $k \in  \{1,\dots,N \}$, 
$$  \lambda(\mathcal{V}_{k}(x)) \geq \frac{\alpha b}{8M} \lambda(\mathcal{V}_{k-1}(x)) .$$
\end{lemma}
\begin{proof}
For all $k\in \{1,\ldots, N\}$ and all $x\in [0,1]^d$, we have that $P_n^X(\mathcal V_k(x)) \geq \alpha P_n^X(\mathcal{V}_{k-1}(x))$. 
Applying Lemma \ref{lemmaWager0}, we obtain with probability $1-2\delta $, 
for all $k\in \{1,\ldots, N\}$ and all $x\in [0,1]^d$,
$$ 2 P^X(\mathcal V_k(x)) \geq P_n^X(\mathcal V_k(x)) \geq \alpha P_n^X(\mathcal{V}_{k-1}(x)) \geq \frac{\alpha }{4} P^X(\mathcal V_{k-1}(x)).$$
Using the boundedness assumptions on $f_X$, we obtain, for any hyper-rectangle $V\subset [0,1]^d$, $M \lambda(V) \geq  P^X(V ) \geq b \lambda(V)$, which allows to conclude.

\end{proof}

\begin{lemma}\label{reci_wager_new}
Let $\delta \in (0, 1/2)$. Consider a tree of depth $N$ such that there exists $\rho \in (0, 1/2]$ where the relative split positions $(U_k(x))_{i=1,\dots,N}$ satisfy $U_k(x) \in [\rho, 1-\rho]$ for all $k \in \{1,\dots,N\}$ and $x\in [0,1]^d$. Suppose that the density of $X$, $f_X$, is such that $0 < b \leq f_X(x) \leq M < \infty$ for all $x \in [0,1]^d$. Then whenever $16\log( 4 (2n+1)^{2d} / \delta) \leq n (b\rho / M ) ^{N}$, the tree is, with probability at least $1-2\delta$, $(\tilde{\alpha},P_n^X)$-regular with $\tilde{\alpha} = {b\rho}/{(8M)}$.
\end{lemma}

\begin{proof}
    % Let $\mathcal{V}$ be a node of the tree, and $\mathcal{V}_L, \mathcal{V}_R$ its children resulting from a split at depth $i \in \{1, \dots, N\}$. Let $\mathcal{W} = \{\mathcal{V}_L, \mathcal{V}_R\}$ denote the pair of children of the cell $\mathcal{V}$. 
    By construction, the split position $U_k(x)$ along the chosen direction $D_i$ determines the volumes of the children. Let $\lambda$ denote the Lebesgue measure. Then, we have, for all $k\in\{1,\ldots, N\}$,
    $$\lambda(\mathcal{V}_{k}(x)) = U_k(x)  \lambda(\mathcal{V}_{k-1}(x)) \quad \text{or} \quad \lambda(\mathcal{V}_{k}(x)) = (1-U_k(x))  \lambda(\mathcal{V}_{k-1}(x)) , $$
depending on $x$ positioning with respect to the split $U_k(x) $.
 Since $U_k(x) \geq \rho $ and $1- U_k(x) \geq \rho $, it follows that, for all $k\in\{1,\ldots, N\}$, 
    $$\lambda(\mathcal{V}_{k+1}(x)) \geq \rho  \lambda(\mathcal{V}_{k}(x)).$$
By assumption on $f_X$, for any Borel set $A \subset S_X$, we have $b \lambda(A) \leq P^X(A) \leq M \lambda(A)$. Applying this to the parent and children nodes, we find, for all $x\in [0,1]^d $, for all $k \in \{1,\ldots, N \}$,
$$P^X(\mathcal{V}_{k}(x)) \geq b \lambda(\mathcal{V}_{k}(x) ) \geq b \rho \lambda(\mathcal{V}_{k-1}(x)) \geq \frac{b \rho}{M} P^X(\mathcal{V}_{k-1}(x)).$$
Consequently, the tree is $( b\rho / M , P^X)$-regular. Since $16\log( 4 (2n+1)^{2d} / \delta) \leq n (b\rho / M ) ^{N}$, we can apply Lemma \ref{lemmaWager1} to obtain that with probability $1-2\delta$, for all $x\in [0,1]^d $, for all $k \in \{1,\ldots, N \}$,
$$2 P^X_n(\mathcal{V}_{k}(x))  \geq P^X(\mathcal{V}_{k}(x))  \geq \frac{b \rho}{M} P^X(\mathcal{V}_{k-1}(x))\geq \frac{b \rho}{4M} P^X_n(\mathcal{V}_{k-1}(x)) .$$
\end{proof}

Let us state the following Vapnik-type inequality \cite{vapnik2015uniform}, which involves some standard-error normalization. The first inequality in the next theorem is Theorem 2.1 in  \cite{anthony1993result} (see also Theorem 1.11 in  \cite{lugosi2002pattern}). The second inequality can be obtained from the first one. For more details, one can also refer to the book by \cite{boucheron2013concentration}, especially chapters 12 and 13, as well as \cite{devroye96probabilistic}.

\begin{theorem}[normalized Vapnik inequality] \label{th:vapnik_normalized}
 Let $(Z, Z_1, \ldots , Z_n) $ is a collection of random variables independent and identically distributed with common distribution $P^Z$ on $(S,\mathcal S)$. For any $A \in \mathcal{S},$ let denote $nP_n^Z(A) = \sum_{i=1}^n \ind_{A}(Z_i)$. For any class $\mathcal A\subset \mathcal S $, $\delta > 0$ and $n\geq 1$, it holds with probability at least $1 -\delta $, for all $A\in \mathcal A$,
$$ P_n^Z(A) \geq P^Z(A) \left(1-\sqrt { \frac{4 \log( 4\mathbb S_\mathcal A(2n) / \delta) }{ n P^Z(A)} } \right).$$ 
In particular, with probability at least $1-\delta$ we have, for all $A\in \mathcal A$,
$$   P^Z(A) \leq \dfrac{4}{n} \log\left(\dfrac{4\, \mathbb S_\mathcal A(2n)}{\delta} \right) + 2 P_n^Z(A).$$ 
In addition, we have with probability at least $1-\delta$, for all $A\in \mathcal A$,
$$  P^Z(A) \geq P^Z_n (A) \left(1-\sqrt { \frac{4 \log( 4 \mathbb S_\mathcal A(2n) / \delta) }{ nP^Z_n (A)} } \right)  . $$
\end{theorem}

\begin{proof}
    The first statement is proved in \cite{anthony1993result}.   Let us  prove the second statement. According to the first point, with probability at least $1 - \delta$, we have for all $A \in \mathcal{A}$ $$n P_n^Z(A) - n P^Z(A) \geq     - \sqrt{  4nP^Z(A)    \log(4\mathbb S_{\mathcal A}(2n) /\delta )   },$$   equivalently,  
    $$nP^Z(A) - \sqrt{  4nP^Z(A)    \log(4\mathbb S_{\mathcal A}(2n) /\delta )   } -  n P_n^Z(A)\leq 0.$$
    Setting $x = \sqrt{nP^Z(A)}$, $\alpha = \sqrt{  4 \log(4\mathbb S_{\mathcal A}(2n) /\delta )}$ and $\beta = nP_n^Z(A)$, we have that $x^2 - \alpha x - \beta \leq 0.$ Solving the inequality, we find  $$( \alpha - \sqrt{\alpha^2 + 4\beta}) / 2 \leq x \leq ( \alpha + \sqrt{\alpha^2 + 4\beta} ) / 2.$$ 
    Since $x$ is positive, squaring both sides yields $x^2 \leq ( \alpha + \sqrt{\alpha^2 + 4\beta})^2 / 4$. Then, by the inequality \((a+b)^2 \leq 2(a^2 + b^2)\), it follows that $nP^Z(A) = x^2 \leq \alpha^2 + 2\beta = 4 \log(4\mathbb S_{\mathcal A}(2n) /\delta ) + 2nP_n^Z(A)$
    which is the desired result by dividing each side of the inequality by \(n\).   
\end{proof}

%\begin{theorem}[Vapnik inequality]
%Let $(Z, Z_1, \ldots , Z_n) $ is a collection of random variables independent and identically distributed with common distribution $P$ on $(S,\mathcal S)$. For any Borelian class $\mathcal A\subset \mathcal S $, $\delta > 0$ and $n\geq 1$, it holds with probability at least $1 -\delta $ :
%$$ \forall B \in \mathcal A, \quad \left|\PP_n(B) - \PP(B)  \right|   \leq \sqrt { 8 \log( 2 \mathbb S_{\mathcal A} (n)  / \delta ) / n}  ,$$ 
%where $\PP_n(B) := \sum_{i=1} ^ n  \ind _ B (Z_i) / n$  is the empirical measure and $\mathbb S_{\mathcal A} (n)$ is the shattering coefficient.
%\end{theorem}

%The proof can be found in \cite{hagerup1990guided}. 
%\vspace{5pt}
%\begin{theorem}\label{lemma=chernoff2}
%Let $(Z_i)_{i\geq 1}$ be a sequence of independent random variables with common distribution $B (\mu)$, $\mu \in(0,1) $.   For any $\eta \in (0,1)$ and all $n\geq 1$, we have
%\begin{align*}
 %\PP \left(  \sum_{i=1} ^n Z_i  <  (1-  \eta ) n\mu   \right) \leq \exp( - \eta^2 n \mu / 2)    .
%\end{align*}
 %In addition, we have
%\begin{align*}
 % \PP  \left(  \sum_{i=1} ^n Z_i  > (1+  \eta ) n\mu   \right) \leq \exp( - \eta^2 n \mu  / 3).
%\end{align*}
%\end{theorem}

The following result is standard and known as the multiplicative Chernoff bound for empirical processes. The following version can be found in \cite{hagerup1990guided}.

\begin{theorem}\label{lemma=chernoff}
Let $(Z, Z_1, \ldots , Z_n) $ is a collection of random variables independent and identically distributed with common distribution $P^Z$ on $(S,\mathcal S)$. Let $A$ be a set in $\mathbb R^d$ and let denote $nP_n^Z(A) = \sum_{i=1}^n \ind_{A}(Z_i)$. For any $\delta \in (0,1)$ and all $n\geq 1$, we have with probability at least $1-\delta$
\begin{align*}
P_n^Z(A) \geq \left(1- \sqrt{ \frac{2 \log(1/\delta)  }{ nP^Z(A) } } \right) P^Z(A)  .
\end{align*}
 In addition, for any $\delta \in (0,1)$ and $n\geq 1$, we have with probability at least $1-\delta$
\begin{align*}
P_n^Z(A)  \leq \left(1 +  \sqrt{ \frac{3 \log(1/\delta)   }{ n P^Z(A)} }  \right) P^Z(A).
\end{align*}
\end{theorem}

\end{document}